\documentclass[reqno]{amsart}

\usepackage{enumitem}
\usepackage{amssymb,amsmath,mathtools}
\usepackage{hyperref}
\usepackage{graphicx}
\usepackage{tikz}

\usepackage{microtype}
\usepackage{bbm}

\newcommand*{\mailto}[1]{\href{mailto:#1}{\nolinkurl{#1}}}

\newcommand{\msc}[1]{\href{http://www.ams.org/msc/msc2010.html?t=&s=#1}{#1}}

\newtheorem{theorem}{Theorem}[section]
\newtheorem{lemma}[theorem]{Lemma}

\newtheorem{remark}[theorem]{Remark}

\newcommand{\R}{{\mathbb R}}
\newcommand{\N}{{\mathbb N}}
\newcommand{\C}{{\mathbb C}}

\newcommand{\OO}{{\mathcal O}}

\newcommand{\loc}{\mathrm{loc}}

\newcommand{\E}{\mathrm{e}}

\DeclareMathOperator{\re}{Re}

\newcommand{\nn}{\nonumber}
\newcommand{\be}{\begin{equation}}
\newcommand{\ee}{\end{equation}}

\newcommand{\ti}{\tilde}
\newcommand{\abs}[1]{\left\lvert #1 \right\rvert}

\newcommand{\norm}[1]{\left\lVert #1 \right\rVert}

\newcommand{\ve}{\varepsilon}

\numberwithin{equation}{section}

\newcommand{\hyp}[5]{\,\mbox{}_{#1}F_{#2}\!\left(
  \genfrac{}{}{0pt}{}{#3}{#4};#5\right)}
  
\makeatletter
\newcommand{\dlmf}[1]{%
\cite[%
  \def\nextitem{\def\nextitem{, }}%
  \@for \el:=#1\do{\nextitem\href{http://dlmf.nist.gov/\el}{(\el)}}%
]{dlmf}%
}
\makeatother

\begin{document}
\title[Tranformation Operators]{Transformation operators for spherical Schr\"odinger operators}

\author[M. Holzleitner]{Markus Holzleitner}
\email{\mailto{markus.holzleitner1@gmail.com}}



\thanks{{\it Research supported by the Austrian Science Fund (FWF) 
under Grant No.\ W1245.}}


\begin{abstract}
The present work aims at obtaining estimates for transformation operators for one-dimensional perturbed radial Schr\"odinger operators. It provides more details and suitable extensions to already existing results, that are needed in other recent contributions dealing with these kinds of operators.
\end{abstract}

\maketitle

\section{Introduction}
In general, transformation and transmutation for one dimensional Schr\"odinger or Sturm-Liouville operators on the whole or the half line have a long history due to their importance in inverse spectral theory, see e.g. \cite[Page 145--163]{bos} for an overview. 
The present work deals with transformation properties for the radial Schr\"odinger operators

\begin{equation} \label{Schr}
  H:= - \frac{d^2}{dx^2} + \frac{l(l+1)}{x^2} + q(x):=H_l+q,\quad x\in \R_+:=(0,\infty),
\end{equation}
where $l \geq -\frac12$ and $q$ should satisfy some further integrability conditions mentioned later. Operators of the form \eqref{Schr} appear naturally in higher dimensional models after a separation of variables, and therefore have received considerable attention (see, e.g., \cite{ckt1}, \cite{ckt2}, \cite{cfh}, \cite{kst}, \cite{kt}, \cite{kt2}, \cite{ktc}, \cite{kty}, \cite[Section 3.7]{ok} and \cite{wdln}). It's also worthwhile mentioning, that one field of recent research is concerned about proving dispersive estimates for the related Schr\"odinger equations, c.f. \cite{HKT}, \cite{HKT2}, \cite{ktt} and \cite{kotr}. In many of these contributions the existence and precise estimates for transformation operators for $H$ are crucial. There are some rather old publications available, that aim at proving these properties for $H$: \cite{volk} is concerned with transformation operators near $0$, and \cite{soh} with the situation near $\infty$, cf. also \cite{cc, HM}. Unfortunately, we realized, that these results don't cover all the situations that are considered in the recent articles mentioned before; thus the aim of the present work is to fill this gap, i.e. to give full and detailed proofs and also to provide appropriate extensions. The work should also be seen as a useful tool to stimulate further research for topics that deal with Bessel operators of the form $H$. Now let us discuss the main theorems that we want to establish.  By $\tau$, $\tau_l$ let us denote the differential symbols corresponding to $H$, $H_l$ respectively. 
We first focus on transformation near $0$: The intention is to construct a transformation operator, that maps a solution
$\phi_l(k^2,x)$, $k^2 \in \C_+$, of the equation

\be \label{eq:Besselfree}
\tau_l \phi_l(k^2,x)=k^2 \phi_l(k^2,x), 
\ee 
to a solution $\phi(k^2,x)$ of
\be \label{eq:Bessel}
\tau \phi(k^2,x)=k^2 \phi(k^2,x), 
\ee 
such that the properties of $\phi_l$ near $x=0$ are preserved. Concerning the asymptotic behavior of these solutions $\phi_l$, we refer e.g. to \cite[Section 2]{HKT2, ktt}. We want to express this transformation operator as an integral operator and prove an estimate for it. To fix some notation, for any Lebesgue measurable set $A\subset  \R_+$, $L^p (A, w(z))$ denotes the usual weighted $L^p$ space with positive weight $w(z)$. 
Furthermore, by $p'$ we denote the corresponding dual index, i.e. $\frac1p+\frac{1}{p'}=1$.
The main theorem of the first section now reads as follows:
\begin{theorem}\label{mainthm1}
Let $L>0$ fixed and $0<y<x\le L$. Then 
\be\label{eq:to_GL}
\phi(k^2,x) = \phi_l(k^2,x) + \int_0^x B(x,y) \phi_l(k^2,y) dy =: (I+B)\phi_l(k^2,x),
\ee
where $B:\R_+^2\to \R$ is the so-called Gelfand--Levitan kernel. Concerning the conditions on the potential $q$ for \eqref{eq:to_GL} to hold, and the estimates for $B$, we need to distinguish three cases:
\begin{enumerate}
\item If $l\geq0$, $p>1$ and  $q\in L^p_{\loc} ([0,\infty))$, then \eqref{eq:to_GL} is valid and $B$ satisfies the following estimate:
\be \label{eq:GLest}
\abs{B(x,y)} \le \frac{(xy)^{\frac{1}{2p'}-\alpha}}{\alpha} \norm{q}_{L^p ((0,x])} x^{2\alpha} \exp \left(\frac{\ti{C} x^{1+\frac{1}{p'}}}{\alpha} \norm{q}_{L^p ((0,x])}\right),
\ee
where $0<\alpha<\frac{1}{2p'}$.
\item If $-\frac12<l<0$, $p>\frac{1}{2l+1}$ and $q\in L^p_{\loc} ([0,\infty))$, then \eqref{eq:to_GL} and \eqref{eq:GLest} are valid for $0<\alpha<l+\frac{1}{2p'}$. 
\item If $l=-\frac12$, $2<p<\infty$ and $q\in L^p_{\loc} ([0,\infty), z^{-\frac{p}{p'}})$, then \eqref{eq:to_GL} and 
\begin{align} 
\abs{B(x,y)} &\le \frac{(xy)^{\frac{1}{p'}-\alpha}}{\alpha} \norm{q}_{L^p ( 0,x], z^{-\frac{p}{p'}})} x^{2\alpha} (\max(1,L))^{\frac{1}{2p'}} \nn \\
&\times \exp \left(\frac{\ti{C} (\max(1,L))^{\frac{1}{2p'}} x^{1+\frac{1}{p'}}}{\alpha} \norm{q}_{L^p ((0,x], z^{-\frac{p}{p'}})} \right) \label{eq:GLest-12}
\end{align}
hold, where $0<\alpha<-\frac12+\frac{1}{p'}$.
\end{enumerate}
The constant $\ti{C}$ depends on $l$. 
\end{theorem}
An important conclusion of this theorem is, that the closer the parameter $l$ is to $-\frac12$, the more we need to restrict our assumptions on the potential $q$. Moreover, in the case $l=-\frac12$, not even boundedness of $q$ seems to be enough to guarantee the desired estimates. 
The proof of this result will be discussed in the first section.
The aim of the second section is to verify a similar result near $\infty$, i.e. establishing the following theorem, where $f(k,x)$, $f_l(k,x)$ denote the Jost solutions of the corresponding equations \eqref{eq:Besselfree}, \eqref{eq:Bessel} respectively, which satisfy $f(k,x)\sim \mathrm{e}^{\mathrm{i} kx}$ as $x \to \infty$(for details, again cf. \cite[Section 2]{HKT2, ktt}):

\begin{theorem}\label{mainthm2}
Suppose $\int_1^\infty (x+x^{l})|q(x)| dx <\infty$. Introduce
\[
\ti{\sigma}_j(x):=\int_x^\infty y^j |q(y)|dy.
\]
and let $x<y<\infty$ and $0<\beta\le \frac12$. Then 
\be\label{eq:to_Mar}
f(k,x) = f_l(k,x) + \int_x^\infty K(x,y) f_l(k,y) dy =: (I+K)f_l(k,x),
\ee
where the so-called Marchenko kernel $K:\R^2\to \R$ satisfies the following estimates:
\begin{enumerate}
\item If $l>-\frac12$, then
\[
|K(x,y)| \le C_l \left(\frac{2}{x}\right)^{l}\ti{\sigma}_0\left(\frac{x+y}{2}\right)\exp \left(C_l [\ti{\sigma}_1(x)-\ti{\sigma}_1(\frac{x+y}{2})]\right).
\]
\item If $l=-\frac12$, then
\be
|K(x,y)| \le \frac{C_{-\frac12}}{\beta} \left(\frac{2}{x}\right)^{-\frac12+\beta}\ti{\sigma}_0\left(\frac{x+y}{2}\right)\exp \left(\frac{C_{-\frac12}}{\beta}[\ti{\sigma}_1(x)-\ti{\sigma}_1(\frac{x+y}{2})]\right).
\label{eq:MAest}
\ee 
\end{enumerate}
\end{theorem}
Here we end up with a similar situation as in the case of Theorem \ref{mainthm1}: The bigger the parameter $l$, the more restrictive the assumptions on $q$ need to be. The approach we use to obtain our results is in principle well known: first of all, one establishes a second order equation for the kernel, which can be solved using Riemann's method in combination with successive approximation. The crucial points are the estimates for the Riemann function and the iterates, which we improve at some points. Let us finish the introduction by briefly explaining the main novelties of this article: concerning the transformation operators near $0$, we are able to generalize the previous results from \cite{fad,cc, HM, volk}, where only continuous potentials $q\in C[0,L]$ were considered. 
Moreover we are able to fix some technical inconsistencies in the proofs of the estimates for $B$ and provide further details to make the presentation more accessible. It should also be mentioned that in \cite{HM}, by using a different method, more general classes of potentials could be included, however, only existence of the transformation operators was established, without providing explicit estimates.
For the transformation operators near $\infty$, the results in \cite{soh} only consider estimates for the case $l>0$, which we generalize to $-\frac12 \le l$. 
\section{Transformation Operators near $0$}
\label{Isection}
As a starting point, we want to obtain an equation for the kernel $B(x,y)$ on a finite interval $0<y\le x \le L$. To this end we assume first that $B$ is $C^2(\R^2_+)$ and satisfies the estimates from Theorem \ref{mainthm1}, thus 
\be \label{eq:B2bcnew}
B(x,y)=
\begin{cases}
	\OO(y^{\frac{1}{2p'}-\alpha}), & l>-\frac12 \\
	\OO(y^{\frac{1}{p'}-\alpha}), & l=-\frac12 \\
\end{cases}
\quad \text{ as } y \to 0.
\ee
We start by differentiating \eqref{eq:to_GL} twice with respect to $x$ to obtain

\begin{align} \label{eq:phi1}
\phi''(z,x)=&\phi''_l(z,x)+\frac{\partial B(x,x)}{\partial x} \phi_l(z,x)+B(x,x)\phi'_l(z,x) \\
&+\frac{\partial B(x,y)}{\partial x} \Big| _{y=x} \phi_l(z,x) +\int_{0}^{x} \frac{\partial^2 B(x,y)}{\partial x^2}\phi_l(z,y) dy. \nonumber
\end{align}
On the other hand, using the facts that $\phi$ satisfies \eqref{eq:Bessel}, $\phi_l$ satisfies \eqref{eq:Besselfree} and plugging in \eqref{eq:to_GL} for $\phi$ , we also get
\be \label{eq:phi2}
\phi''(z,x)=\phi_l''(z,x)+q(x)\phi_l(z,x)+\int_{0}^{x} B(x,y)(\frac{l(l+1)}{x^2} + q(x)-z)\phi_l(z,y) dy.
\ee
Once more applying \eqref{eq:Besselfree} and integrating by parts twice leads to
\begin{align} \label{eq:phi3}
z \int_{0}^{x} B(x,y) \phi_l(z,y) dy=&\int_{0}^{x}B(x,y)\frac{l(l+1)}{y^2}\phi_l(z,y)dy+B(x,0)\phi'_l(z,0) \nonumber \\
&-B(x,x)\phi'_l(z,x)+\frac{\partial B(x,y)}{\partial y} \Big|_{y=x} \phi_l(z,x) \nonumber\\
&-\frac{\partial B(x,y)}{\partial y} \Big|_{y=0} \phi_l(z,0)-\int_{0}^{x} \frac{\partial^2 B(x,y)}{\partial y^2}\phi_l(z,y) dy.
\end{align}
Now plugging \eqref{eq:phi3} into \eqref{eq:phi2} and setting \eqref{eq:phi1} equal to \eqref{eq:phi2} gives us the following identity for the kernel $B(x,y)$:
\begin{align} \label{eq:B1}
&\frac{\partial B(x,x)}{\partial x}\phi_l(z,x)+\left( \frac{\partial B(x,y)}{\partial x}+\frac{\partial B(x,y)}{\partial y} \right)\Big|_{y=x} \phi_l(z,x) \nonumber-q(x)\phi_l(z,x)\\
&+B(x,0)\phi'_l(z,0)-\frac{\partial B(x,y)}{\partial y}\Big|_{y=0}\phi_l(z,0) \nonumber\\
 &+ \int_{0}^{x}\left( \frac{\partial^2 B}{\partial x^2} - \frac{\partial^2 B}{\partial y^2} + \frac{l(l+1)}{y^2}B - \frac{l(l+1)}{x^2}B-q(x) \right) \phi_l(z,y)dy=0
\end{align}
Hence, in order to ensure that the right-hand side of \eqref{eq:to_GL} satisfies equation \eqref{eq:Bessel}, it's sufficient that $B$ solves the following problem:
\begin{align} \label{eq:B2}
&\left( \frac{\partial^2 }{\partial x^2} - \frac{\partial^2 }{\partial y^2}  +  \frac{l(l+1)}{y^2} - \frac{l(l+1)}{x^2}-q(x) \right)B(x,y)=0, \quad 0<y<x \\
&\quad \frac{\partial B(x,x)}{\partial x}=\frac{q(x)}{2}, \quad \lim_{y \to 0} B(x,y) \phi'_l(x,y)=0=\lim_{y \to 0} B(x,y) y^{l}. \label{eq:B2bc}
\end{align}
The term $\frac{\partial B(x,y)}{\partial y}\Big|_{y=0}\phi_l(z,0)$ disappears, since $\phi_l(z,y)=\OO(y^{l+1})$ by the properties of $\phi_l$ mentioned e.g. in \cite[Section 2]{HKT2, ktt}, and $\frac{\partial B(x,y)}{\partial y}$ can be assumed to be bounded(cf. Lemma \ref{lem:Bder}). 
The next step is to bring \eqref{eq:B2}--\eqref{eq:B2bc} into a simpler form. 

Let
\begin{align} \label{eq:Btr}
z:=\frac{(x+y)^2}{4}, \quad s:=\frac{(x-y)^2}{4}, \quad \text{ and } u(z,s)=(z-s)^{l} B(x,y). 
\end{align}
A straightforward calculation yields the following equation for $u(z,s)$:
\be \label{eq:u}
\frac{\partial^2 u}{\partial z \partial s}+\frac{l}{z-s}\frac{\partial u}{\partial z}-\frac{l}{z-s}\frac{\partial u}{\partial s}=\frac{1}{4 \sqrt{zs}}q(\sqrt{z}+\sqrt{s})u,
\ee
whereas the boundary conditions \eqref{eq:B2bc} transform according to
\be \label{eq:ubc}
\frac{\partial u}{\partial z}+\frac{lu}{z}=\frac{z^{l-\frac12}q(\sqrt{z})}{4}, \quad u(z,z)=
\begin{cases}
\OO(z^{l+\frac{1}{2p'}-\alpha}), & l>-\frac12 \\
\OO(z^{l+\frac{1}{p'}-\alpha}), & l=-\frac12 \\
\end{cases}
\quad \text{ as } z \to 0. 
\ee
To solve the equation \eqref{eq:u}--\eqref{eq:ubc}, we will use Riemann's method, a well known approach to treat linear hyperbolic partial differential equations of second order in two independent variables. For further information and applications we refer to the huge amount of literature, e.g. \cite{ch2}, \cite{ksg}, \cite{lev}. 
Let us continue by introducing the following operator defined on $C^2(\R^2_+)$:

\be\label{def:L}
Lu:= \frac{\partial^2 u}{\partial z \partial s}+\frac{l}{z-s}\frac{\partial u}{\partial z}-\frac{l}{z-s}\frac{\partial u}{\partial s}
\ee
and its formal adjoint, which can be computed, using integration by parts, as
\[
Mv:=L^{*}v=\frac{\partial^2 v}{\partial z \partial s}-\frac{l}{z-s}\frac{\partial v}{\partial z}+\frac{l}{z-s}\frac{\partial v}{\partial s}-\frac{2l}{(z-s)^2}v.
\]
Next let $0<\eta \le \xi \le L$ and $0<\ve<\delta$. We define the points $0'$, $A$, $B$, $B'$, $B_1$, $B_2$, $B_3$, $B_4$, $C$, $C_1$, $C_2$ and $P$ in the $z$-$s$-plane according to the following picture:

\begin{tikzpicture}[scale=5]
    \draw [|-|,thick] (0,2) node (saxis) [above] {$s$}
        |- (2,0) node (zaxis) [right] {$z$};
		\coordinate (C) at (1.1,0) ;
		\coordinate (B_4) at (1.2,0.9) ;
		\coordinate (B') at (1.1,0.9) ;
    \draw (0,0) -- (1.1,1.1) coordinate (B) node[above] {$B$};
    \draw (0.2,0) coordinate (a_3) -- (1.3,1.1) coordinate (B_3) node[above] {$B_3$};
		\draw (C) node[below] {$C$}-- (B);
		\draw (1.2,0) coordinate (C_2) node[below] {$C_2$} -- (1.2,1) coordinate (B_2) node[below,right] {$B_2$};
		\draw (1,0) coordinate (C_1) node[below] {$C_1$} -- (1,0.8) coordinate (B_1) node[below,left] {$B_1$};
		\draw (1.8,0) coordinate (A) node[below] {$A$} -- (1.8,1.1) coordinate (P) node[right,above] {$P$};
		\draw (1.1,1.1) coordinate (B) -- (1.8,1.1) coordinate (P);
    \fill[black] (C) circle (0.4pt);
		\fill[black] (C_1) circle (0.4pt);
		\fill[black] (C_2) circle (0.4pt);
		\fill[black] (A) circle (0.4pt);
		\fill[black] (B) circle (0.4pt);
		\fill[black] (P) circle (0.4pt);
		\fill[black] (B') circle (0.4pt) node[left] {$B'$};
    \fill[black] (B_1) circle (0.4pt);
		\fill[black] (B_2) circle (0.4pt);
		\fill[black] (B_3) circle (0.4pt);
		\fill[black] (B_4) circle (0.4pt) node[below,right] {$B_4$};
		\fill[black] (0.2,0) circle (0.4pt) node[below] {$0'$};
		\fill[black] (0,0) circle (0.4pt) node[below] {$0$};
    \draw [<->](0,-0.15) -- (0.2,-0.15) node[sloped,below,midway] {$\delta$};
		\draw [<->](0,-0.25) -- (1.1,-0.25) node[sloped,below,midway] {$\eta$};
		\draw [<->](0,-0.35) -- (1.8,-0.35) node[sloped,below,midway] {$\xi$};
		\draw [<->](1.35,1) -- (1.35,0.9) node[right,midway] {$\varepsilon$};
		\draw [<->](1,0.5) -- (1.1,0.5) node[sloped,below,midway] {$\varepsilon$};
		\draw [<->](1.1,0.4) -- (1.2,0.4) node[sloped,below,midway] {$\varepsilon$};
		\draw [<->](1.9,0) -- (1.9,1.1) node[right,midway] {$\eta$};
		
\end{tikzpicture}

By $G$ let us denote the region enclosed by the segment $\Gamma:=\overline{A P B 0}$. If all the appearing functions are smooth enough and well defined on $\overline{G}$, applying Green's Theorem leads to:
\begin{align} \nonumber
2 \iint_G \left( vLu-uMv \right) dz ds=& \oint_{\Gamma} \left( u\frac{\partial v}{\partial z}- v \frac{\partial u}{\partial z}+\frac{2l}{z-s}uv \right) dz\\
&- \left( u\frac{\partial v}{\partial s}- v \frac{\partial u}{\partial s}-\frac{2l}{z-s}uv \right) ds. \label{eq:Green}
\end{align}
We will choose $v$ such that the above formula simplifies.
However, the problem is, that $z=s$ and $\eta=z$ might lead to singularities for our chosen $v$, so we have to be careful and continue as follows: we divide $G$ into the parts $APB_3B_2C_2$ and $0'B_1 C_1$ and investigate these regions separately, thus isolating the singularities at $z=s$ and $\eta=z$, and afterward let $\ve$ and $\delta$ tend to $0$. Let's begin with $APB_3B_2C_2$ and apply Green's Theorem to get:

\begin{align} 
2 \iint_{APB_3B_2C_2} \Big( vLu&-uMv \Big) dz ds= \nonumber \\
\int_{C_2 A} \left( u\frac{\partial v}{\partial z}- v \frac{\partial u}{\partial z}+\frac{2l}{z}uv \right) dz &- \int_{AP}\left( u\frac{\partial v}{\partial s}- v \frac{\partial u}{\partial s}-\frac{2l}{\xi-s}uv \right) ds \nonumber \\
+ \int_{PB_3} \left( u\frac{\partial v}{\partial z}- v \frac{\partial u}{\partial z}+\frac{2l}{z-\eta}uv \right)dz &- \int_{B_2 C_2} \left( u\frac{\partial v}{\partial s}- v \frac{\partial u}{\partial s}-\frac{2l}{\eta+\ve-s}uv \right) ds \nonumber \\
+\int_{B_3 B_2} \Big( u \frac{\partial v}{\partial z}- v \frac{\partial u}{\partial z} &+ \frac{4l}{z-s}uv- u\frac{\partial v}{\partial s}+ v \frac{\partial u}{\partial s} \Big) dz. \label{eq:Green2}
\end{align}
Now we further evaluate some of the appearing integrals using integration by parts:
\begin{align}
\int_{C_2 A} \left( u\frac{\partial v}{\partial z}- v \frac{\partial u}{\partial z}+\frac{2l}{z}uv \right) dz&=uv \Big|_{C_2}^{A}-2 \int_{C_2 A} v \left( \frac{\partial u}{\partial z} -\frac{l}{z}u \right) dz \nonumber \\
\int_{AP}\left( u\frac{\partial v}{\partial s}- v \frac{\partial u}{\partial s}-\frac{2l}{\xi-s}uv \right) ds&= -uv \Big|_{A}^{P}+2 \int_{AP} u \left( \frac{\partial v}{\partial s} -\frac{l}{\xi-s}v \right) ds \nonumber \\
\int_{PB_3} \left( u\frac{\partial v}{\partial z}- v \frac{\partial u}{\partial z}+\frac{2l}{z-\eta}uv \right)dz&= -uv \Big|_{P}^{B_3} +2 \int_{P B_3} u \left( \frac{\partial v}{\partial z} +\frac{l}{z-\eta}v \right) ds. \label{eq:Green3}
\end{align}
To proceed, we introduce the Riemann-Green function $v_1$ 
as a solution to the following problem:
\begin{align}
	1)& \quad Mv_1=0 \nonumber \\
	2)& \quad \frac{\partial v_1}{\partial z} + \frac{l}{z-\eta} v_1=0 \quad \text{on } PB \nonumber\\
	3)& \quad \frac{\partial v_1}{\partial s} - \frac{l}{\xi-s} v_1=0. \quad \text{on } AP \nonumber\\
	4)& \quad v_1(P)=1 \label{eq:propv_1}
\end{align}
In \cite{ksg}, \cite{lev} an explicit formula was computed:
\be \label{eq:formulav_1}
v_1(z,s; \eta, \xi)=(\eta-z)^l (s-\xi)^l (s-z)^{-2l}\hyp21{-l,-l}{1}{\frac{(z-\xi)(s-\eta)}{(z-\eta)(s-\xi)}}.
\ee
The proof of this formula heavily relies on the connection between symmetry groups for certain PDEs and special functions. For further information on this topic we refer e.g. to \cite{mil1}--\cite{mil2}.
To continue, employing \eqref{eq:propv_1} and \eqref{eq:Green3} leads to:

\begin{align}
2 \iint_{APB_3B_2C_2}  v_1Lu dz ds&= -2 \int_{C_2 A} v_1 \left( \frac{\partial u}{\partial z} -\frac{l}{z}u \right) dz +uv_1 \Big|_{C_2}^{A}+uv_1 \Big|_{A}^{P}-uv_1 \Big|_{P}^{B_3} \nonumber \\
&- \int_{B_2 C_2} \left( u\frac{\partial v_1}{\partial s}- v_1 \frac{\partial u}{\partial s}-\frac{2l}{\eta+\ve-s}uv_1 \right) ds \nonumber \\
+\int_{B_3 B_2} &\Big( u \frac{\partial v_1}{\partial z}- v_1 \frac{\partial u}{\partial z} + \frac{4l}{z-s}uv_1- u\frac{\partial v_1}{\partial s}+ v_1 \frac{\partial u}{\partial s} \Big) dz. \label{eq:Green4}
\end{align}
Next let's focus on the region enclosed by $0' B_1 C_1$. Similar considerations as before imply:
\begin{align}
2 \iint_{0' B_1 C_1} \Big( vLu&-uMv \Big) dz ds=uv \Big|_{0'}^{C_1}-2 \int_{0' C_1} v \left( \frac{\partial u}{\partial z} -\frac{l}{z}u \right) dz \nonumber \\
&- \int_{C_1 B_1} \left( u\frac{\partial v}{\partial s}- v \frac{\partial u}{\partial s}-\frac{2l}{\eta-\ve-s}uv \right) ds \nonumber \\
+\int_{B_1 0'} &\Big( u \frac{\partial v}{\partial z}- v \frac{\partial u}{\partial z} + \frac{4l}{z-s}uv- u\frac{\partial v}{\partial s}+ v \frac{\partial u}{\partial s} \Big) dz. \label{eq:Green5}
\end{align}
Using formula \eqref{eq:formulav_1} for $v_1$ and the form of a second linearly independent solution to \eqref{eq:hypeq}, we obtain a solution $v_2$, defined in $0' B_1 C_1$, to the following problem:
\begin{align}
	1)& \quad Mv_2=0 \nonumber \\
	2)& \quad v_2(z,z)=0 \nonumber \\
	3)& \quad v_1(z,\eta)-v_2(z,\eta)=\OO(1) \quad \text{as} \quad z \to \eta. \label{eq:propv_2}
\end{align}
$v_2$ also has an explicit representation(cf. \cite{lev} for details):
\begin{align}
v_2(z,s; \eta, \xi)=(-1)^{l}\frac{\Gamma(1+l)}{\Gamma(-l) \Gamma(2+2l)} &(z-s) (\eta-\xi)^{1+2l}(\eta-z)^{-l-1} (s-\xi)^{-l-1} \nonumber \\
\times &\hyp21{1+l,1+l}{2+2l}{\frac{(z-s)(\eta-\xi)}{(z-\eta)(s-\xi)}}. \label{eq:formulav_2}
\end{align}
Setting $v=v_2$ in \eqref{eq:Green5} we end up with
\begin{align}
2 \iint_{0' B_1 C_1} & v_2Lu dz ds=uv \Big|_{0'}^{C_1}-2 \int_{0' C_1} v_2 \left( \frac{\partial u}{\partial z} -\frac{l}{z}u \right) dz \nonumber \\
&- \int_{C_1 B_1} \left( u\frac{\partial v_2}{\partial s}- v_2 \frac{\partial u}{\partial s}-\frac{2l}{\eta-\ve-s}uv_2 \right) ds \nonumber \\
+\int_{B_1 0'} &\Big( u \frac{\partial v_2}{\partial z}- v_2 \frac{\partial u}{\partial z} + \frac{4l}{z-s}uv_2- u\frac{\partial v_2}{\partial s}+ v_2 \frac{\partial u}{\partial s} \Big) dz. \label{eq:Green6}
\end{align}
Now if we introduce
\[
v=\begin{cases} v_1, & z>\eta,\\ v_2, & z<\eta.\end{cases}
\]
and combine \eqref{eq:Green4}, \eqref{eq:Green6} with the boundary condition \eqref{eq:ubc}, this implies
\begin{align}
2 \iint_{0' B_1 C_1+APB_3B_2C_2}  vLu  dz ds=-\frac12 &\int_{0'C_1}v q(\sqrt{z}) z^{l-\frac12} dz - \frac12 \int_{C_2 A} v q(\sqrt{z}) z^{l-\frac12} dz \nonumber \\
+uv_1 \Big|_{C_2}^{A}+uv_1 \Big|_{A}^{P}&-uv_1 \Big|_{P}^{B_3}+uv \Big|_{0'}^{C_1}+\Delta_1+\Delta_2, \label{eq:Green7}
\end{align}

where

\begin{align}
\Delta_1 &:=\int_{C_2 B_2} \left( u\frac{\partial v_1}{\partial s}- v_1 \frac{\partial u}{\partial s}-\frac{2l}{\eta+\ve-s}uv_1 \right) ds \nonumber \\
&-\int_{C_1 B_1} \left( u\frac{\partial v_2}{\partial s}- v_2 \frac{\partial u}{\partial s}-\frac{2l}{\eta-\ve-s}uv_2 \right) ds \label{eq:delta1},
\end{align}

\begin{align}
 \Delta_2 &:=\int_{B_3 B_2}\Big( u \frac{\partial v_1}{\partial z}- v_1 \frac{\partial u}{\partial z} +\frac{4l}{z-s}uv_1- u\frac{\partial v_1}{\partial s}+ v_1 \frac{\partial u}{\partial s} \Big) dz \nn \\
&+\int_{B_1 0'} \Big( u \frac{\partial v_2}{\partial z}- v_2 \frac{\partial u}{\partial z} + \frac{4l}{z-s}uv_2- u\frac{\partial v_2}{\partial s}+v_2 \frac{\partial u}{\partial s} \Big) dz. \label{eq:delta2}
\end{align}
Before performing the limits $\ve$ and $\delta \to 0$ in \eqref{eq:Green7}, let us provide some auxiliary estimates for $v$. They have basically already been obtained in \cite{kty}, however, we will give a slightly more general version here and, for the reader's convenience, also repeat the main arguments:
\begin{lemma} \label{lem:vest} \cite[Lemma A.2]{kty}
Fix some $0\le s \le \eta$ and let $z_1(s)$, $z_2(s)$ be defined by the following equations: 

\be \label{def:z12}
-1=\frac{(z_1(s)-\xi)(s-\eta)}{(z_1(s)-\eta)(s-\xi)}, \quad -1=\frac{(z_2(s)-s)(\eta-\xi)}{(z_2(s)-\eta)(s-\xi)}
\ee

Then the functions $v_1$ and $v_2$ satisfy

\begin{align}
\abs{v_1(z,s;\eta,\xi)} &\le C_1 (z-\eta)^{l} (\xi-s)^{l} (z-s)^{-2l}, \quad z \in (z_1, \xi) \label{eq:estv_11} \\
\abs{v_1(z,s;\eta,\xi)} &\le C_2 (z-s)^{-2l} (\xi-z)^{l} (\eta-s)^l (\log {   \frac{(\xi-z)(\eta-s)}{(z-\eta)(\xi-s)} } +1), \quad  z \in (\eta, z_1) \label{eq:estv_12} \\
\abs{v_2(z,s;\eta,\xi)} &\le C_3 (\xi-\eta)^{1+2l} (z-s) (\xi-s)^{-l-1} (\eta-z)^{-l-1}, \quad z \in (0,z_2) \label{eq:estv_21} \\
\abs{v_2(z,s;\eta,\xi)} &\le C_4 (\xi-\eta)^{l} (z-s)^{-l} (\log {   \frac{(z-s)(\xi-\eta)}{(\eta-z)(\xi-s)} } +1), \quad z \in (z_2, \eta). \label{eq:estv_22}
\end{align}

\end{lemma}

\begin{proof}
The proof heavily relies on estimates for the hypergeometric function, collected in Appendix \ref{sec:hypgeom}. Let's denote the argument $\frac{(z-\xi)(s-\eta)}{(z-\eta)(s-\xi)}$ of the hypergeometric function in \eqref{eq:formulav_1} by $\sigma_1$. We start by proving \eqref{eq:estv_11}. In this case we have that $0\geq\sigma_1 \geq -1$ and $\hyp21{-l,-l}{1}{\sigma_1}$ is bounded, since for $k>l$, the expression $\abs{\sigma_1}^k \left( \frac{(-l)_k}{k!}\right)^2$ is monotone decreasing and converges to $0$, thus the series in \eqref{Gausshyp} is converging.
For \eqref{eq:estv_12}, we first consider the case $l \notin \N_{0}$. Noting that $\sigma_1 \le -1$ and employing \eqref{eq:Fnearinfty} and \eqref{eq:psiref}, we end up with the following equation:

\begin{align*}
\hyp21{-l,-l}{1}{\sigma_1}=\frac{ (-\sigma_1)^{l}}{\Gamma(-l)\Gamma(1+l)} &\sum_{k=0}^{\infty}\left( \frac{(-l)_k}{k!}\right)^2 \sigma_1^{-k} \\
&\times ( \log(-\sigma_1)+2\psi(1+k)-2\psi(k-l)+\pi \cot{\pi l} ).
\end{align*}
By \eqref{eq:psinearinfty}, we see that $\displaystyle \psi(1+k)-\psi(k-l)=\frac{l+1}{k-l}+\OO(k^{-2})$ and thus 
\[ 
\sum_{k=0}^{\infty}\left( \frac{(-l)_k}{k!}\right)^2 \sigma_1^{-k}( 2\psi(1+k)-2\psi(k-l))
\] 
admits $\displaystyle \sum_{k=0}^{\infty}  \frac{l+1}{k-l} \left( \frac{(-l)_k}{k!}\right)^2$ as a uniform bound. Therefore we can deduce:

\begin{align*}
\abs{\hyp21{-l,-l}{1}{\sigma_1}}=&\frac{\abs{\sigma_1}^l}{\abs{\Gamma(-l)\Gamma(1+l)}} \big| ( \log(-\sigma_1)+\pi \cot{\pi l} )\sum_{k=0}^{\infty}\left( \frac{(-l)_k}{k!}\right)^2 \sigma_1^{-k} \\
&+\sum_{k=0}^{\infty}\left( \frac{(-l)_k}{k!}\right)^2 \sigma_1^{-k}(2\psi(1+k)-2\psi(k-l)) \big| \\
&\le \frac{\abs{\sigma_1}^l}{\abs{\Gamma(-l)\Gamma(1+l)}} (C_1 \abs{\log(-\sigma_1)+\pi \cot{\pi l}} +C_2) \\
&\le C_2 \abs{\sigma_1}^l (\log(-\sigma_1)+1)
\end{align*}
In the case $l \in \N$, the hypergeometric function reduces to a polynomial and thus the proof is easy. The proof of the remaining estimates \eqref{eq:estv_21}--\eqref{eq:estv_22} is similar.
\end{proof}
We continue our investigation of formula \eqref{eq:Green7} with the following lemma:
\begin{lemma} \label{lem:delta1}
If $\ve \to 0$, we obtain
\be
\Delta_1 \to (A_2-A_1) \left( \frac{\xi-\eta}{\eta-s} \right)^{l} u \Big|_{C}^{B'} \label{eq:delta1lim}
\ee 
for some $A_1, A_2 \in \R$. 
\end{lemma}

\begin{proof}
We integrate by parts to obtain:

\begin{align}
\Delta_1=uv_1 \Big|_{C_2}^{B_2} - &uv_2 \Big|_{C_1}^{B_1}-2 \int_{C_2 B_2} v_1 \left( \frac{\partial u}{\partial s}+ \frac{l}{\eta+\ve-s}u \right) ds\nn \\
+&2 \int_{C_1 B_1} v_2 \left( \frac{\partial u}{\partial s}+ \frac{l}{\eta-\ve-s}u \right) ds. \label{eq:Delta1part1}
\end{align}
Using that $u\in C^2(\overline{G})$, and thus $\frac{\partial u}{\partial s}(.,s)$, $u(.,s)$ being locally Lipschitz continuous, we observe the following properties:
\begin{align}
\frac{\partial u}{\partial s} \Big|_{z=\eta-\ve}-\frac{\partial u}{\partial s} \Big|_{z=\eta+\ve}=\OO(\ve), \quad \frac{lu(\eta-\ve,s)}{\eta-\ve-s}-\frac{lu(\eta+\ve,s)}{\eta+\ve-s}=\OO(\ve). \label{eq:ueps}
\end{align}
Inserting \eqref{eq:ueps} into \eqref{eq:Delta1part1}, we obtain the following expression for 
$\Delta_1$:
\begin{align*}
uv_1 \Big|_{C_2}^{B_2} - uv_2 \Big|_{C_1}^{B_1} &+2\int_{0}^{\eta-\ve-\delta} \left( v_2(\eta-\ve,s)-v_1(\eta+\ve,s \right) \left(\frac{\partial u}{\partial s}+\frac{lu}{\eta-\ve-s}\right)\Big|_{z=\eta-\ve}ds \\
&-2 \int_{B_4 B_2} v_1 \left( \frac{\partial u}{\partial s} + \frac{l}{\eta+\ve-s}u \right) ds+\OO(\ve).
\end{align*}
Now we use \eqref{eq:formulav_1}, \eqref{eq:formulav_2} and \eqref{eq:Fnearinfty} to infer:
\[
v_1(\eta+\ve,s)=\frac{1}{\Gamma(-l)\Gamma(l+1)}\left( \frac{\xi-\eta}{\eta-s} \right)^{l} \cdot \log\left(\frac{(\eta-s)(\xi-\eta)}{\ve(\xi-s)} \right)+A_1 \left( \frac{\xi-\eta}{\eta-s} \right)^{l}+\OO(\ve),
\]

\[
v_2(\eta-\ve,s)=\frac{1}{\Gamma(-l)\Gamma(l+1)}\left( \frac{\xi-\eta}{\eta-s} \right)^{l} \cdot \log\left(\frac{(\eta-s)(\xi-\eta)}{\ve(\xi-s)} \right)+A_2 \left( \frac{\xi-\eta}{\eta-s} \right)^{l}+\OO(\ve).
\]
Thus we end up with
\begin{align*}
\Delta_1=uv_1 \Big|_{C_2}^{B_2} - uv_2 \Big|_{C_1}^{B_1}+2 \int_{C_1 B_1} \OO(\ve) &\left( \frac{\partial u}{\partial s}+\frac{lu}{\eta-\ve-s} \right) ds \\
+2 \int_{C_1 B_1}(A_2-A_1) \left( \frac{\xi-\eta}{\eta-s} \right)^{l} \left( \frac{\partial u}{\partial s}+\frac{lu}{\eta-\ve-s} \right) ds&+ \int_{B_4 B_2} \OO( \log(\ve) ) ds + \OO(\ve).
\end{align*}
Finally we let $\ve \to 0$, so that one more integration by parts leads us to:
\begin{align*}
\Delta_1=&uv_1 \Big|_{C_2 \to C}^{B_2 \to B} + uv_2 \Big|_{B_1 \to B}^{C_1 \to C} +2(A_2-A_1) \left( \frac{\xi-\eta}{\eta-s} \right)^{l} u \Big|_{C}^{B'} \\
+2&\int_{CB'} (A_2-A_1) u  \left( -\frac{\partial}{\partial s} \left( \frac{\xi-\eta}{\eta-s} \right)^{l} + \frac{l}{\eta-s} \left( \frac{\xi-\eta}{\eta-s} \right)^{l} \right) ds.
\end{align*}
Since the integral expression disappears and by observing that 
\[
(uv_2)(C_1)-(uv_1)(C_2) \xrightarrow{\ve \to 0} (A_2-A_1)\left( \frac{\xi-\eta}{\eta-s} \right)^{l} u(C)
\]
and
\[
(uv_1)(B_2)-(uv_2)(B_1) \xrightarrow{\ve \to 0} (A_1-A_2)\left( \frac{\xi-\eta}{\eta-s} \right)^{l} u(B'),
\]
(here we again use the asymptotics for $v_1$ and $v_2$ and the fact that the log-terms cancel), the claim follows.
\end{proof}

As a consequence, Lemma \ref{lem:delta1} leads to the following expression for \eqref{eq:Green7}, when we let $\ve \to 0$:

\begin{align}
2 \iint_{0' B_3 P A}  vLu  dz ds=-\frac12 &\int_{0'A}v q(\sqrt{z}) z^{l-\frac12} dz 
+2u(P)-u(0')v(0')-u(B_3)v(B_3) \nn \\
&+u(B')(A_2-A_1) \left( \frac{\xi-\eta}{\delta} \right)^{l}+\Delta_2. \label{eq:Green8}
\end{align}

It remains to perform the limit $\delta \to 0$. In the next lemma this is done for $\Delta_2$:
\begin{lemma} \label{lem:delta2}
If $\delta \to 0$, we obtain $\Delta_2 \to 0$.
\end{lemma}
\begin{proof}
Here we only sketch the proof in a way such that the main argument should be clear.
Let's first divide the integral into two parts:

\begin{align*}
\Delta_2 &=\int_{B_3 B'}\Big( u \frac{\partial v_1}{\partial z}- v_1 \frac{\partial u}{\partial z} +\frac{4l}{z-s}uv_1- u\frac{\partial v_1}{\partial s}- v_1 \frac{\partial u}{\partial s} \Big) dz \nn \\
&+\int_{B' 0'} \Big( u \frac{\partial v_2}{\partial z}- v_2 \frac{\partial u}{\partial z} + \frac{4l}{z-s}uv_2- u\frac{\partial v_2}{\partial s}- v_2 \frac{\partial u}{\partial s} \Big) dz=:\Delta_{2,1}+\Delta_{2,2}
\end{align*}
and focus on the part $\Delta_{2,1}$ first. Each summand in this expression needs to be treated separately, however, since the calculations are similar, we will only focus on $v_1 \frac{\partial u}{\partial z}$. Without loss of generality, we can also set $\ve=\frac{\delta}{2}$, compute the integral along $\overline{B_3 B_2}$ and then let $\delta \to 0$. An integration by parts gives: 
\begin{align*}
\int_{B_3 B_2} v_1 \frac{\partial u}{\partial z} ds= (u v_1) \big|_{B_3}^{B_2}-\int_0^{\frac{\delta}{2}} (u \frac{\partial v_1}{\partial z})(\eta+\frac{\delta}{2}+t,\eta-\frac{\delta}{2}+t) dt.
\end{align*}
Next we provide an estimate for $\frac{\partial v_1}{\partial z}$ along $\overline{B_3 B_2}$. A straightforward calculation using \eqref{eq:Gausshypder} gives:
\begin{align}
\frac{\partial v_1}{\partial z}&=-l(\eta-z)^{l-1} (s-\xi)^l (s-z)^{-2l}\hyp21{-l,-l}{1}{\frac{(z-\xi)(s-\eta)}{(z-\eta)(s-\xi)}} \nn \\
&+2l(\eta-z)^l (s-\xi)^{l} (s-z)^{-2l-1}\hyp21{-l,-l}{1}{\frac{(z-\xi)(s-\eta)}{(z-\eta)(s-\xi)}} \nn \\
&+(\eta-z)^l (s-\xi)^{l} (s-z)^{-2l} (-l)^2 \frac{(s-\eta)(\xi-\eta)}{(s-\xi)(z-\eta)^2}  \nn \\
&\times \hyp21{-l+1,-l+1}{2}{\frac{(z-\xi)(s-\eta)}{(z-\eta)(s-\xi)}}, \label{est:v_1der}
\end{align}
and thus, since 
\[
\sup_{(z,s) \in \overline{B_3 B_2}}\abs{\frac{(z-\xi)(s-\eta)}{(z-\eta)(s-\xi)}}=\sup_{z \in (\eta+\frac{\delta}{2}, \eta+\delta)} \abs{\frac{(z-\xi)(z-\delta-\eta)}{(z-\eta)(z-\delta-\xi)}}<C,
\]
the hypergeometric-function-terms appearing in \eqref{est:v_1der} are bounded uniformly(cf. the calculations in Lemma \ref{lem:vest} and Appendix \ref{sec:hypgeom} ), which results in
\[
\sup_{(z,s) \in \overline{B_3 B_2}} \abs{\frac{\partial v_1}{\partial z}(z,s)}\le  C_l \delta^{-l-1} 
\]
Consequently, in combination with \eqref{eq:ubc} and the fact, that the length of $\overline{B_3 B_2}$ is proportional to $\delta$, we obtain:
\[
\abs{\int_{B_3 B_2} v_1 \frac{\partial u}{\partial z} ds}= 
\begin{cases}
\OO(\delta^{\frac{1}{2p'}-\alpha}), & l>-\frac12, \\
\OO(\delta^{\frac{1}{p'}-\alpha} ), & l=-\frac12, 
\end{cases}
\]
which goes to $0$ as $\delta \to 0$.  
For $\Delta_{2,2}$, we will again only have a look at the term $ \frac{4l}{z-s}uv_2$, since similar considerations also apply for the other ones. Using \eqref{eq:formulav_2} and \eqref{eq:ubc} in combination with the observation, that the hypergeometric-function-term in \eqref{eq:formulav_2} is bounded for sufficiently small $\delta>0$, leads us to:
\[
\abs{\int_{B_3 B'} \frac{4l}{z-s}uv_2 dz}
\]
\[
\le
\begin{cases}
C_l \delta^{l+\frac{1}{2p'}-\alpha} \int_{\delta}^{\eta-\frac{\delta}{2}}(\eta-z)^{-l-1}dz \le C_l  \delta^{l+\frac{1}{2p'}-\alpha} \delta^{-l} , &l \geq 0, \\
C_0 \delta^{\frac{1}{2p'}-\alpha} \int_{\delta}^{\eta-\frac{\delta}{2}}(\eta-z)^{-1}dz \le C_0  \delta^{\frac{1}{2p'}-\alpha} (-\log(\delta)) , &l = 0, \\
C_l \delta^{l+\frac{1}{2p'}-\alpha} \int_{\delta}^{\eta-\frac{\delta}{2}}(\eta-z)^{-l-1}dz \le C_l  \delta^{l+\frac{1}{2p'}-\alpha} (\eta-\frac{\delta}{2})^{-l}, &0>l>-\frac12, \\
C_{-\frac12} \delta^{-\frac12+\frac{1}{p'}-\alpha} \int_{\delta}^{\eta-\frac{\delta}{2}}(\eta-z)^{-\frac12}dz \le C_{-\frac12}  \delta^{-\frac12+\frac{1}{p'}-\alpha} (\eta-\frac{\delta}{2})^{\frac12}, &l=-\frac12.
\end{cases}
\]
Letting $\delta$ tend to $0$ finishes the proof of the lemma.
\end{proof}


If we let $\delta$ tend to $0$ and apply Lemma \ref{lem:delta2}, the expression \eqref{eq:Green8} provides us the following integral equation for $u$: 

\be \label{eq:intequ}
u(\xi, \eta)= \frac14 \int_{0}^{\xi} v q(\sqrt{z}) z^{l-\frac12} dz + \frac14 \iint_{0 B P A}  (zs)^{-\frac12} vq(\sqrt{z}+\sqrt{s})u  dz ds
\ee
Now we want to show that under our assumptions from Theorem \ref{mainthm1}, this equation has indeed a solution. To do so we use the successive approximation method, which will lead to the subsequent result:
\begin{theorem} \label{thm:estu}
Under the conditions on $q$ stated in Theorem \ref{mainthm1}, there is a unique continuous function  $u(\xi,\eta)$ that solves \eqref{eq:intequ} and satisfies 
\begin{align} \nn
\abs{u(\xi,\eta)}&\le\frac{(\xi-\eta)^{l+\frac{1}{2p'}-\alpha}}{\alpha} (\sqrt{\xi}+\sqrt{\eta})^{2\alpha} \norm{q}_{L^p ((0,\sqrt{\xi}+\sqrt{\eta}])}  \\
	&\times \exp \left(\frac{\ti{C}(\sqrt{\xi}+\sqrt{\eta})^{1+\frac{1}{p'}}}{\alpha} \norm{q}_{L^p ((0,\sqrt{\xi}+\sqrt{\eta}])}\right),\quad  l>-\frac12, \nn \\
\abs{u(\xi,\eta)}&\le\frac{(\xi-\eta)^{l+\frac{1}{p'}-\alpha}}{\alpha} (\sqrt{\xi}+\sqrt{\eta})^{2\alpha} (\max(1,L))^{\frac{1}{2p'}} \norm{q}_{L^p( (0,\sqrt{\xi}+\sqrt{ \eta} ], z^{-\frac{p}{p'}} )}\nn\\
&\times \exp \left(\frac{\ti{C}(\sqrt{\xi}+\sqrt{\eta})^{1+\frac{1}{p'}}(\max(1,L))^{\frac{1}{2p'}}}{\alpha} \norm{q}_{L^p( (0,\sqrt{\xi}+\sqrt{ \eta} ], z^{-\frac{p}{p'}} )} \right), \quad l=-\frac12. \label{eq:uest}
\end{align}
The constant $\ti{C}$ depends on $l$.
\end{theorem}

We intend to represent $u$ as a series $u=u_0+\sum_{n=1}^{\infty} u_n$, where the $u_n$'s are defined recursively as follows:

\begin{align} 
u_0(\xi,\eta)&:= \frac14 \int_{0}^{\xi} v(z,0) q(\sqrt{z}) z^{l-\frac12} dz \nn \\
u_{n+1}(\xi,\eta)&:= \frac14 \iint_{0 B P A}  (zs)^{-\frac12} v(z,s)q(\sqrt{z}+\sqrt{s})u_n(z,s)  dz ds. \label{def:un}
\end{align}
The crucial point here is to find appropriate estimates for the iterates, such that the series that defines $u$ converges. This will be done carefully in several steps.
We start with the following lemma, which will act as a useful tool to estimate certain integral expressions:
\begin{lemma} \label{lem:logint}
Let $\gamma>1$ and $0\le \ti{z}\le 1$. Then we have:
\be
\norm{\log(z)}_{L^{\gamma}((0,\ti{z}])} \le C \gamma,
\ee
where the constant $C$ is independent from $\overline{z}$.
\end{lemma}
\begin{proof}
We first use the transformation $u=-\log(z)$ to get
\[
\int_{0}^{\ti{z}}(-\log(z))^{\gamma} dz= \int_{-\log(\ti{z})}^{\infty} u^{\gamma} \E^{-u} du =\Gamma(\gamma+1, -\log(\ti{z})),
\]
where $\Gamma(a, z)$ denotes the incomplete Gamma function, cf. \dlmf{8.2.2}. Furthermore $\Gamma(a, z)$ enjoys the following asymptotics(cf. \dlmf{8.11.2}):
\[
\Gamma(a, z)=z^{a-1} \E^{-z}+ \OO(z^{-1}), \quad z \to \infty,
\]
which lead to the estimate
\[
\norm{\log(z)}_{L^{\gamma}((0,\ti{z}])} \le C (-\log(\ti{z})) \ti{z}^{\frac{1}{\gamma}} \le C\gamma,
\]
where in the last inequality we used that the local maximum of $(-\log(\ti{z})) \ti{z}^{\frac{1}{\gamma}}$ on $[0,1]$ is a multiple of $\gamma$.
\end{proof}

In the next three lemmas we investigate the iterates $u_n$ and thus start with $u_0$:

\begin{lemma} \label{lem:u_0est}
The following estimates hold:

\begin{align} \label{eq:estu_0}
\abs{u_0(\xi,\eta)} &\le \frac{C}{\alpha} \norm{q}_{L^p( (0,\sqrt{\xi}+\sqrt{ \eta} ] )}(\sqrt{\xi}+\sqrt{\eta})^{2\alpha}(\xi-\eta)^{l+\frac{1}{2p'}-\alpha}, \quad  l>-\frac12 \\
\abs{u_0(\xi,\eta)} &\le \frac{C}{\alpha} \norm{q}_{L^p( (0,\sqrt{\xi}+\sqrt{ \eta} ], z^{-\frac{p}{p'}} )}(\max(1,L))^{\frac{1}{2p'}}(\sqrt{\xi}+\sqrt{\eta})^{2\alpha}\\
&\times (\xi-\eta)^{-\frac12+\frac{1}{p'}-\alpha}, \quad l=-\frac12, \nn
\end{align}

where the constant $C$ depends on $l$.
\end{lemma}

\begin{proof}
First we split the integral for $u_0$ into two parts:
\[
\frac14 \int_{0}^{\xi} v(z,0) q(\sqrt{z}) z^{l-\frac12} dz=\frac14 \int_{0}^{\eta} v(z,0) q(\sqrt{z}) z^{l-\frac12} dz+\frac14 \int_{\eta}^{\xi} v(z,0) q(\sqrt{z}) z^{l-\frac12} dz
\]	
and estimate each part separately. We then estimate the integrals from $\eta$ to $\xi$ and from $0$ to $\eta$ respectively and use \eqref{eq:estv_11}--\eqref{eq:estv_22} to further decompose it into three more parts:

\begin{align*}
& \quad \frac14 \int_{\eta}^{\xi} \abs{v(z,0) q(\sqrt{z})} z^{l-\frac12} dz \le C_2 \eta^l \int_{\eta}^{z_1(0)} \frac{(\xi-z)^l}{z^{l+\frac12}} \abs{q(\sqrt{z})} dz \\
+C_1 \xi^l \int_{z_1(0)}^{\xi} &\frac{(z-\eta)^l}{z^{l+\frac12}} \abs{q(\sqrt{z})} dz
+C_2 \eta^l \int_{\eta}^{z_1(0)} \frac{(\xi-z)^l}{z^{l+\frac12}}  \log{ \left( \frac{(\xi-z) \eta }{(z-\eta) \xi}\right) } \abs{q(\sqrt{z})} dz \\
     & =:I_1+I_2+I_3
\end{align*}
and
\begin{align*}
 \frac14 \int_{0}^{\eta} &\abs{v(z,0) q(\sqrt{z})} z^{l-\frac12} dz \le C_4 \frac{(\xi-\eta)^{1+2l}}{\xi^{1+l}} \int_{0}^{z_2(0)} \frac{z^{l+\frac12}}{(\eta-z)^{l+1}} \abs{q(\sqrt{z})} dz\\
+C_5 (\xi-\eta)^l &\int_{z_2(0)}^{\eta} z^{-\frac12} \abs{q(\sqrt{z})} dz \eta^l+ C_5 (\xi-\eta)^l \int_{z_2(0)}^{\eta} \log{ \left( \frac{z(\xi-\eta) }{(\eta-z) \xi}\right) } \abs{q(\sqrt{z})} dz \\
&=:I_4+I_5+I_6,
\end{align*}
where $\displaystyle z_1(0)$ and $z_2(0)$ are given via \eqref{def:z12}. To bound $I_1$, we first note that
$\displaystyle \frac{\xi-\eta}{2\eta} \le \frac{\xi-z}{z} \le \frac{\xi-\eta}{\eta}$, i.e. $ \displaystyle \frac{(\xi-z)^l}{z^l} \le C_l \frac{(\xi-\eta)^l}{\eta^l}$, where $C_l=\max(1,2^{-l})$. 
Hence in the case $l>-\frac12$:
\begin{align*}
\abs{I_1} &\le 2C_l C_2 (\xi-\eta)^l \int_{\eta}^{z_1(0)}z^{-\frac12} \abs{q(\sqrt{z})} dz\\
&\le 2C_l C_2 (\xi-\eta)^l  \norm{q}_{L^p( (0,\sqrt{\xi} ] )} \left( \int_{\sqrt{\eta}}^{\sqrt{z_1(0)}} 1 dz \right)^{\frac{1}{p'}}\\
&\le \ 2 C_l C_2 (\xi-\eta)^l\left(\frac{\eta}{\xi+\eta}(\xi-\eta) \right)^{\frac{1}{2p'}} \norm{q}_{L^p( (0,\sqrt{\xi} ] )} \\
&\le \ 2C_l C_2 (\sqrt{\xi}+\sqrt{\eta})^{2\alpha}(\xi-\eta)^{l+\frac{1}{2p'}-\alpha} \norm{q}_{L^p( (0,\sqrt{\xi}+\sqrt{ \eta} ] )},
\end{align*}
where we have used H\"older's inequality, the elementary estimate $\sqrt{a-b} \geq \sqrt{a}-\sqrt{b}$ for $a\geq b$ and the fact that $\displaystyle \left( \frac{(\sqrt{\xi}+\sqrt{\eta})^2}{\xi-\eta} \right)^{\alpha} \geq 1$. In the case $l=-\frac12$ we proceed as follows:
\begin{align*}
\abs{I_1} &\le 2C_{-\frac12} C_2 (\xi-\eta)^{-\frac12}  \int_{\sqrt{\eta}}^{\sqrt{z_1(0)}} z^{\frac{1}{p'}} z^{-\frac{1}{p'}} \abs{q(z)} dz \\
&\le 2C_{-\frac12} C_2 (\xi-\eta)^{-\frac12} \left( \int_{\sqrt{\eta}}^{\sqrt{z_1(0)}} z dz \right)^{\frac{1}{p'}} \norm{ q}_{L^p( (0,\sqrt{\xi} ] , z^{-\frac{p}{p'}})} \\
&\le 2C_{-\frac12} C_2 (\sqrt{\xi}+\sqrt{\eta})^{2 \alpha} (\xi-\eta)^{-\frac12+\frac{1}{p'}-\alpha}  \norm{ q}_{L^p( (0,\sqrt{\xi}+\sqrt{ \eta} ], z^{-\frac{p}{p'}})}.
\end{align*}
The calculations for $I_2$ are similar, we just use $ \displaystyle \frac{(z-\eta)^l}{z^l} \le C_l \frac{(\xi-\eta)^l}{\xi^l}$ instead. Let's continue with $I_3$, again in the case $l>-\frac12$ first:
\begin{align*}
&\abs{I_3} \le C_l C_1 (\xi-\eta)^l \int_{\eta}^{z_1(0)} z^{-\frac12} \log\left(\frac{\eta}{z-\eta}\right)  \abs{q(\sqrt{z})} dz \\
&=2C_l C_1 (\xi-\eta)^l \int_{\sqrt{ \eta}}^{\sqrt{z_1(0)}} \log\left(\frac{\eta}{z^2-\eta}\right) \abs{q(z)} dz \\
&\le 2C_l C_1 (\xi-\eta)^l \int_{\sqrt{ \eta}}^{\sqrt{z_1(0)}} \log\left(\frac{\sqrt{\eta}}{z-\sqrt{\eta}}\right) \abs{q(z)} dz \\
&\le 2C_l C_1 (\xi-\eta)^l \norm{q}_{L^p( (0,\sqrt{\xi} ] )} \left(  \int_{\sqrt{ \eta}}^{\sqrt{z_1(0)}} dz \right)^{\frac{1}{p'}-2\alpha} \\
&\times \left( \int_{\sqrt{ \eta}}^{\sqrt{z_1(0)}} \log^{\frac{1}{2\alpha}}\left(\frac{\sqrt{\eta}}{z-\sqrt{\eta}}\right)dz \right)^{2\alpha} \\
&\le 2C_l C_1 (\xi-\eta)^l (\xi-\eta)^{\frac{1}{2p'}-\alpha} \left( \sqrt{\eta} \int_{0}^{\sqrt{\frac{\xi-\eta}{\xi+\eta}}} (-\log(u))^{\frac{1}{2\alpha}}du\right)^{2\alpha} \norm{q}_{L^p( (0,\sqrt{\xi} ] )} \\
&\le 2C_l C_1 \frac{(\sqrt{\xi}+\sqrt{\eta})^{2\alpha}}{ \alpha}(\xi-\eta)^{l+\frac{1}{2p'}-\alpha}  \norm{q}_{L^p( (0,\sqrt{\xi}] )}, 
\end{align*}
where in the fourth step we used H\"older's inequality with indices $\frac{p'}{1-2\alpha p'}$, $\frac{1}{2\alpha}$ and $p$, in the fifth step we did a linear transformation inside the logarithmic integral,  
and in the penultimate step we applied Lemma \ref{lem:logint}.
In the case $l=-\frac12$, we have to make similar changes as for $I_1$, namely the fourth step will read as follows:
\[
\left(  \int_{\sqrt{ \eta}}^{\sqrt{z_1(0)}} z^{\frac{1}{p'}\frac{p'}{1-p'\alpha}}dz \right)^{\frac{1}{p'}-\alpha} \left( \int_{\sqrt{ \eta}}^{\sqrt{z_1(0)}} \log^{\frac{1}{\alpha}}\left(\frac{\sqrt{\eta}}{z-\sqrt{\eta}}\right)dz \right)^{\alpha}\norm{ q}_{L^p( (0,\sqrt{\xi} ] , z^{-\frac{p}{p'}})},
\]
while the first integral can be further estimated:
\[
\left(  \int_{\sqrt{ \eta}}^{\sqrt{z_1(0)}} z^{\frac{1}{p'}\frac{p'}{1-p'\alpha}}dz \right)^{\frac{1}{p'}-\alpha} \le z_1(0)^{\frac{\alpha}{2}} \left( \int_{\sqrt{\eta}}^{\sqrt{z_1(0)}} z dz \right)^{\frac{1}{p'}-\alpha},
\]
and now we can proceed as for $I_1$.
For $I_4$ in the case $l>-\frac12$ we can now deduce the following inequality:

\begin{align*}
\abs{I_4} &\le C_3 (\xi-\eta)^{1+2l} \xi^{-1-l} \int_{0}^{z_2(0)} z^{l+1-\frac{1}{2p'}} z^{\frac{1}{2p'}-\frac12} (\eta-z)^{-l-1 } \abs{q(\sqrt{z})} dz  \\
 &\le C_3 (\xi-\eta)^{1+2l} \xi^{-1-l} z_2(0)^{l+1-\frac{1}{2p'}} \left( \int_{0}^{z_2(0)} (\eta-z)^{p'(-l-1)}dz \right)^{\frac{1}{p'}} \\
&\times \left( \int_{0}^{z_2(0)} z^{p(\frac{1}{2p'}-\frac12)} \abs{q(\sqrt{z})}^p  dz \right)^{\frac1p} 
\end{align*}
To further estimate this expression, we need to distinguish cases. If $(-l-1)p'+1<0$, we get: 
\begin{align*}
&\abs{I_4} \le \frac{C_3}{\abs{(-l-1)p'+1}^{\frac{1}{p'}}} (\xi-\eta)^{1+2l} \xi^{-1-l} (\xi \eta)^{l+1-\frac{1}{2p'}}\\
&\times (2\xi-\eta)^{-l-1+\frac{1}{2p'}} \eta^{-l-1+\frac{1}{p'}}(\xi-\eta)^{-l-1+\frac{1}{p'}} (2\xi-\eta)^{l+1-\frac{1}{p'}} \norm{q}_{L^p( (0,\sqrt{\xi}] )}\\
&=\frac{C_3}{\abs{(-l-1)p'+1}^{\frac{1}{p'}}} (\xi-\eta)^{l+\frac{1}{2p'}} \left( \frac{\xi-\eta}{2\xi-\eta}\right)^{\frac{1}{2p'}} \left( \frac{\eta}{\xi}\right)^{\frac{1}{2p'}}\norm{q}_{L^p( (0,\sqrt{\xi}] )} \\
&\le \frac{C_3}{\abs{(-l-1)p'+1}^{\frac{1}{p'}}} (\sqrt{\xi}+\sqrt{\eta})^{2\alpha} (\xi-\eta)^{l+\frac{1}{2p'}-\alpha}  \norm{q}_{L^p( (0,\sqrt{\xi}+\sqrt{ \eta} ] )},
\end{align*}
A similar reasoning in the case $(-l-1)p'+1>0$ yields the same result. 
One also has to treat the case $(-l-1)p'+1=0$, but we omit the details here.
Since $p'$ only depends on $l$, we will end up with a constant only depending on $l$. 
In the $-\frac12$-case, with similar changes as in $I_1-I_3$, we get an additional factor $z(0)^\frac{1}{2p'}$.
For the remaining part of this lemma, we will only focus on the computations in the case $l>-\frac12$ in order to avoid writing down the same changes all the time. Next, for $I_5$ we get:

\begin{align*}
\abs{I_5} \le C_4 (\xi-\eta)^l \int_{\sqrt{z_2(0)}}^{\sqrt{\eta}} &\abs{q(z)} dz \le C_4 (\xi-\eta)^l  \norm{q}_{L^p( (0,\sqrt{\xi} ] )} \left( \int_{\sqrt{z_2(0)}}^{\sqrt{\eta}} 1 dz \right)^{\frac{1}{p'}}\\
 &\le \ C_4 (\sqrt{\xi}+\sqrt{\eta})^{2\alpha}(\xi-\eta)^{l+\frac{1}{2p'}-\alpha} \norm{q}_{L^p( (0,\sqrt{\xi}+\sqrt{ \eta} ] )}.
\end{align*}

It remains to look at $I_6$, and with similar arguments as for $I_3$, we obtain:

\begin{align*}
&\abs{I_6} \le C_4 (\xi-\eta)^l \int_{\sqrt{z_2(0)}}^{\sqrt{\eta}} z^{-\frac12} \log \left( \frac{\sqrt{\eta}}{\sqrt{\eta}-z} \right) \abs{q(z)} dz \\
&\le C_4 (\xi-\eta)^l \left(  \int_{\sqrt{ z_2(0)}}^{\sqrt{\eta}} dz \right)^{\frac{1}{p'}-2\alpha} \left( \int_{\sqrt{z_2(0)}}^{\sqrt{\eta}} \log^{\frac{1}{2\alpha}}\left(\frac{\sqrt{\eta}}{\sqrt{\eta}-z}\right)dz \right)^{2\alpha}\norm{q}_{L^p( (0,\sqrt{\xi} ] )}\\
&\le C_4 (\xi-\eta)^l (\xi-\eta)^{\frac{1}{2p'}-\alpha} \left( \sqrt{\eta} \int_{0}^{\sqrt{\frac{\xi-\eta}{2\xi-\eta}}} (-\log(u))^{\frac{1}{2\alpha}}du\right)^{2\alpha} \norm{q}_{L^p( (0,\sqrt{\xi} ] )} \\
&\le \frac{C_4 (\sqrt{\xi}+\sqrt{\eta})^{2\alpha}}{\alpha} (\xi-\eta)^{l+\frac{1}{2p'}-\alpha}  \norm{q}_{L^p( (0,\sqrt{\xi}+\sqrt{ \eta} ] )}.
\end{align*}

\end{proof}
In the next lemma we are concerned with proving an inequality for $u_1$:
\begin{lemma} \label{lem:u_1est}
The following estimates hold:

\begin{align} 
	\abs{u_1(\xi,\eta)} &\le  \frac{\ti{C}}{\alpha^2} \norm{q}_{L^p( (0,\sqrt{\xi}+\sqrt{ \eta} ] )}^2 \nn \\
	&\times (\sqrt{\xi}+\sqrt{\eta})^{1+\frac{1}{p'}+2\alpha}(\xi-\eta)^{l+\frac{1}{2p'}-\alpha},\quad l>-\frac12 \nn \\
	\abs{u_1(\xi,\eta)} &\le \frac{\ti{C}}{\alpha^2} \norm{ q}_{L^p( (0,\sqrt{\xi}+\sqrt{ \eta} ], z^{-\frac{p}{p'}})}^2 (\max(1,L))^{\frac{2}{2p'}} \nn \\
	&\times (\sqrt{\xi}+\sqrt{\eta})^{1+\frac{2}{p'}+2\alpha}(\xi-\eta)^{-\frac12+\frac{1}{p'}-\alpha}, \quad l=-\frac12. \label{eq:estu_1}
\end{align}

The constant $\ti{C}$ depends on $l$ and it may differ from $C$ in Lemma \ref{lem:u_0est}. 
\end{lemma}

\begin{proof}
Similarly as in Lemma \ref{eq:estu_1}, we split the corresponding integral:

\begin{align*}
&\frac14 \iint_{0 B P A}  (zs)^{-\frac12} v(z,s)q(\sqrt{z}+\sqrt{s})u_0(z,s)  dzds=\\
 &\frac14 \int_{0}^{\eta} \int_{\eta}^{\xi} (zs)^{-\frac12} v(z,s)q(\sqrt{z}+\sqrt{s})u_0(z,s)  dzds\\
 +&\frac14 \int_{0}^{\eta} \int_{s}^{\eta} (zs)^{-\frac12} v(z,s)q(\sqrt{z}+\sqrt{s})u_0(z,s)  dzds.
\end{align*}	
Let us denote $\ti{C_i}:=C C_i$, where $C$ is the constant obtained in Lemma \ref{eq:estu_1}, and the $C_i$'s again are taken from Lemma \ref{lem:vest}.
Using the results from Lemma \ref{lem:vest} and Lemma \ref{lem:u_0est} we end up with:

\begin{align*}
&\frac14 \int_{0}^{\eta} \int_{\eta}^{\xi} (zs)^{-\frac12} \abs{v(z,s)q(\sqrt{z}+\sqrt{s})u_0(z,s)}  dzds \\
&\le \ti{C_2} \int_{0}^{\eta} s^{-\frac12} \int_{\eta}^{z_1(s)} z^{-\frac12} \abs{q(\sqrt{z}+\sqrt{s})} (z-s)^{l+\frac{1}{2p'}-\alpha} \\
&\times(\xi-z)^{l} (\eta-s)^l (z-s)^{-2l} dzds \\
&+ \ti{C_1} \int_{0}^{\eta} s^{-\frac12} \int_{z_1(s)}^{\xi} z^{-\frac12} \abs{q(\sqrt{z}+\sqrt{s})} (z-s)^{l+\frac{1}{2p'}-\alpha} \\
&\times(z-\eta)^{l} (\xi-s)^{l} (z-s)^{-2l} dzds\\
&+ \ti{C_2} \int_{0}^{\eta} s^{-\frac12} \int_{\eta}^{z_1(s)} z^{-\frac12} \abs{q(\sqrt{z}+\sqrt{s})} (z-s)^{l+\frac{1}{2p'}-\alpha} (\xi-z)^{l} (\eta-s)^l \\
& \times (z-s)^{-2l} \log {   \frac{(\xi-z)(\eta-s)}{(z-\eta)(\xi-s)} } dz ds =:J_1+J_2+J_3,
\end{align*}

and similarly:
\begin{align*}
&\frac14 \int_{0}^{\eta} \int_{s}^{\eta} (zs)^{-\frac12} \abs{v(z,s)q(\sqrt{z}+\sqrt{s})u_0(z,s)}  dzds \\
&\le \ti{C_3} \int_{0}^{\eta} s^{-\frac12} \int_{s}^{z_2(s)} z^{-\frac12} \abs{q(\sqrt{z}+\sqrt{s})} (z-s)^{l+1+\frac{1}{2p'}-\alpha}\\
&\times  (\xi-\eta)^{1+2l} (\eta-z)^{-l-1} (\xi-s)^{-l-1} dzds \\
&+ \ti{C_4} \int_{0}^{\eta} s^{-\frac12} \int_{z_2(s)}^{\eta} z^{-\frac12} \abs{q(\sqrt{z}+\sqrt{s})} (z-s)^{l+\frac{1}{2p'}-\alpha} (\xi-\eta)^{l} (z-s)^{-l}  dzds\\
&+ \ti{C_4} \int_{0}^{\eta} s^{-\frac12} \int_{\eta}^{z_1(s)} z^{-\frac12} \abs{q(\sqrt{z}+\sqrt{s})} (z-s)^{l+\frac{1}{2p'}-\alpha} (\xi-\eta)^{l} (z-s)^{-l} \\
&\times \log {   \frac{(z-s)(\xi-\eta)}{(\eta-z)(\xi-s)} } dz ds =:J_4+J_5+J_6,
\end{align*}
where $z_1(s)$ and $z_2(s)$ are again given via \eqref{def:z12}. For the brevity of presentation, we omit the details for $J_1$ and $J_2$ and start by considering $J_3$. A similar reasoning as in the Lemma \ref{eq:estu_1} gives $\displaystyle \left( \frac{\xi-z}{z-s}\right)^l \le C_l \left( \frac{\xi-\eta}{\eta-s}\right)^l$.
Next, as for $I_3$, we use H\"older's inequality with indices $\frac{p'}{1-2\alpha p'}$, $\frac{1}{2\alpha}$ and $p$, the inequality $z_1(s)\le z_1(0)$, and Lemma \ref{lem:logint} to arrive at: 
\begin{align*}
\abs{J_3}\le &\ti{C_3} (\xi-\eta)^l \int_{0}^{\eta} (\sqrt{\xi}+\sqrt{s})^{\frac{1}{p'}}s^{-\frac12} \left( \int_{\sqrt{\eta}}^{\sqrt{z_1(s)}}dz \right)^{{\frac{1}{p'}-2\alpha}}  \\
&\times \left( \int_{\sqrt{ \eta}}^{\sqrt{z_1(0)}} \log^{\frac{1}{2\alpha}}\left(\frac{\sqrt{\eta}}{z-\sqrt{\eta}}\right)dz \right)^{2\alpha} \\
&\times \left(\int_{\sqrt{\eta}}^{\sqrt{z_1(s)}} \norm{q}_{L^p((0,z+\sqrt{s}])}^p \abs{q(z+\sqrt{s})}^p dz \right)^{\frac1p}ds \\
& \le \frac{\ti{C_3}}{(1+\frac{1}{p'})\alpha}(\xi-\eta)^{l+\frac{1}{2p'}-\alpha} (\sqrt{\xi}+\sqrt{\eta})^{1+\frac{1}{p'}+2\alpha} \norm{q}_{L^p( (0,\sqrt{\xi}+\sqrt{\eta}])}^2.
\end{align*}
In the $l=-\frac12$-case, we also only remark, that in the end one gets a factor $(\sqrt{\xi}+\sqrt{\eta})^{1+\frac{2}{p'}+2\alpha}$ instead of $(\sqrt{\xi}+\sqrt{\eta})^{1+\frac{1}{p'}+2\alpha}$
From now on we restrict ourselves to provide details only for the case $l>-\frac12$. The tiny modifications in the $l=-\frac12$-case will always be similar to Lemma \ref{lem:u_0est}. We continue with $J_4$, and here, again for brevity, only consider the case $(-l-1)p'+1<0$. We use $ z_2(s)-s=\frac{(\xi-s)(\eta-s)}{2\xi-\eta-s}$ and H\"older's inequality we infer:
\begin{align*}
\abs{J_4} &\le \ti{C_4} (\xi-\eta)^{1+2l} \int_{0}^{\eta} (\sqrt{\xi}+s)^{2\alpha} s^{-\frac12} (\xi-s)^{-l-1} (\xi-s)^{l+1-\alpha} (\eta-s)^{l+1-\alpha} \\
&\times (2\xi-\eta-s)^{-l-1+\alpha} \left(\int_{0}^{z_2(s)} (\eta-z)^{(-l-1)p'}dz \right)^{\frac{1}{p'}} \\
&\times \left( \int_{0}^{z_2(s)} z^{p(\frac{1}{2p'}-\frac12)} \abs{q(\sqrt{z}+\sqrt{s})}^p \norm{q}_{L^p ((0,\sqrt{z}+\sqrt{s}])}^p dz \right)^{\frac1p} ds.
\end{align*}
We further estimate this expression by evaluating the inner integrals. After that, we use $ \eta-z_2(s)=\frac{(\xi-\eta)(\eta-s)}{2\xi-\eta-s}$ and we group the remaining terms in an appropriate way:
\begin{align*}
&\abs{J_4} \le \frac{\ti{C_4}}{\abs{(-l-1)p'+1}^{\frac{1}{p'}}} (\sqrt{\xi}+\sqrt{\eta})^{2\alpha} (\xi-\eta)^{1+2l} \int_{0}^{\eta} s^{-\frac12} (\xi-s)^{-\alpha} (\eta-s)^{l+1-\alpha} \\
&\times (2\xi-\eta-s)^{-l-1+\alpha} (\eta-s)^{-l-1+\frac{1}{p'}} (\xi-\eta)^{-l-1+\frac{1}{p'}}\\
&\times (2\xi-\eta-s)^{l+1-\frac{1}{p'}} \norm{q}_{L^p( (0,\sqrt{\xi}+\sqrt{\eta}])}^2 ds \\
&\le \frac{\ti{C_4}}{\abs{(-l-1)p'+1}^{\frac{1}{p'}}}  (\xi-\eta)^{l+\frac{1}{2p'}-\alpha}(\sqrt{\xi}+\sqrt{\eta})^{2\alpha} \norm{q}_{L^p( (0,\sqrt{\xi}+\sqrt{\eta}])}^2 \\
&\times \int_{0}^{\eta} (\sqrt{\xi}+\sqrt{s})^{\frac{1}{p'}} s^{-\frac12} \left( \frac{\eta-s}{(\sqrt{\xi}+\sqrt{s})^2} \right)^{\frac{1}{2p'}} \left( \frac{\eta-s}{2\xi-\eta-s} \right)^{\frac{1}{2p'}-\alpha} \left( \frac{\xi-\eta}{2\xi-\eta-s}\right)^{\frac{1}{2p'}} ds.
\end{align*}
The last expression immediately leads to:
\[
\abs{J_4} \le \frac{\ti{C_3}}{(1+\frac{1}{p'})\alpha \abs{(-l-1)p'+1}^{\frac{1}{p'}}}(\xi-\eta)^{l+\frac{1}{2p'}-\alpha} (\sqrt{\xi}+\sqrt{\eta})^{1+\frac{1}{p'}+2\alpha} \norm{q}_{L^p( (0,\sqrt{\xi}+\sqrt{\eta}])}^2.
\]
We omit the details for $J_5$ and $J_6$. Concerning $J_6$, we only remark that we follow the same procedure as for $J_3$, at one point though we have to use the estimate $z_2(s) \geq s$ in order to get $s$ as the lower bound of the inner integral.
\end{proof}
The next lemma treats $u_n$:
\begin{lemma} \label{lem:unest}
The following estimates hold:
\begin{align} 
	\abs{u_n(\xi,\eta)} &\le  \frac{\ti{C}^n}{\alpha^{n+1}n!} \norm{q}_{L^p( (0,\sqrt{\xi}+\sqrt{ \eta} ] )}^{n+1} \nn \\
	&\times (\sqrt{\xi}+\sqrt{\eta})^{n(1+\frac{1}{p'})+2\alpha}(\xi-\eta)^{l+\frac{1}{2p'}-\alpha},\quad l>-\frac12 \nn \\
	\abs{u_n(\xi,\eta)} &\le \frac{\ti{C}^n}{\alpha^{n+1}n!} \norm{ q}_{L^p( (0,\sqrt{\xi}+\sqrt{ \eta} ], z^{-\frac{p}{p'}})}^{n+1} (\max(1,L))^{\frac{n+1}{2p'}} \nn \\
	&\times (\sqrt{\xi}+\sqrt{\eta})^{n(1+\frac{1}{p'})+2\alpha}(\xi-\eta)^{-\frac12+\frac{1}{p'}-\alpha}, \quad l=-\frac12. \label{eq:estu_n}
\end{align}
The constant is identical to the one obtained in Lemma \ref{lem:u_1est}.
\end{lemma}
\begin{proof}
We do a similar integral splitting as before, and, as an example, only provide details for the inequality for $J_3^{n}$. We will proceed inductively:
\begin{align*}
\abs{J_3^{n+1}} &\le \frac{\ti{C}^n C_3 (\xi-\eta)^l}{\alpha^{n+1} n!} \int_{0}^{\eta} s^{-\frac12} \int_{\sqrt{\eta}}^{\sqrt{z_1(s)}} \log\left(\frac{\sqrt{\eta}}{z-\sqrt{\eta}}\right) \\
&\times \abs{q(z+\sqrt{s})} \norm{q}_{L^p ((0,z+\sqrt{s}])}^{n+1}(\sqrt{z}+\sqrt{s})^{n(1+\frac{1}{p'})}(z-s)^{\frac{1}{2p'}} dzds \\
&\le \frac{\ti{C}^n C_3 (\xi-\eta)^l}{\alpha^{n+1} n!} \int_{0}^{\eta} (\sqrt{\xi}+\sqrt{s})^{n(1+\frac{1}{p'})+\frac{1}{p'}}s^{-\frac12} \left( \int_{\sqrt{\eta}}^{\sqrt{z_1(s)}}dz \right)^{{\frac{1}{p'}-2\alpha}}  \\
&\times \left( \int_{\sqrt{ \eta}}^{\sqrt{z_1(0)}} \log^{\frac{1}{2\alpha}}\left(\frac{\sqrt{\eta}}{z-\sqrt{\eta}}\right)dz \right)^{2\alpha} \\
&\times \left(\int_{\sqrt{\eta}}^{\sqrt{z_1(s)}} \norm{q}_{L^p((0,z+\sqrt{s}])}^{p(n+1)} \abs{q(z+\sqrt{s})}^p dz \right)^{\frac1p}ds. \\
\end{align*}
With an analogous reasoning as for $J_3$, we obtain
\[
\abs{J_3^{n+1}} \le \frac{\ti{C}^{n+1} (\xi-\eta)^{l+\frac{1}{2p'}-\alpha}}{\alpha^{n+2} (n+1)!} (\sqrt{\xi}+\sqrt{\eta})^{(n+1)(\frac{1}{p'}+1)+2\alpha} \norm{q}_{L^p( (0,\sqrt{\xi}+\sqrt{ \eta} ] )}^{n+2}.
\]
\end{proof}
We are now in the position to finish the proof of Theorem \ref{thm:estu}:
\begin{proof}[Proof of Theorem~\ref{thm:estu}]
Everything now follows from the Lemmas \ref{lem:vest}--\ref{lem:unest}, since $\sum_{n=0}^{\infty} \abs{u_n}$ converges uniformly on compact sets.
\end{proof}
We continue now with some remarks, which aim at relating previously obtained results to this work:
\begin{remark} \label{rem:rem1}
It has already been mentioned in \cite[Appendix]{kty}, that the estimates for $u$ in \cite{volk} contain an error. Indeed, if they were true, we would have the inequality $\abs{B(x,x)}\le C x^{2-2\rho}$ for any $0\le \rho <1$, which is impossible for $\rho<\frac12$ due to $\frac{\partial B(x,x)}{\partial x}=\frac{q(x)}{2}$(\eqref{eq:B2bc}), because not even a constant potential $q(x)=1$ would satisfy the condition. That's the main reason, why we and the authors of \cite{kty} have been very careful in the proofs and also provided many details regarding the technical estimates. Moreover, our computations also allow to generalize to potentials lying in some $L^p$-space.
\end{remark}
\begin{remark} \label{rem:rem2}
In the $-\frac12$ case, however, it seems, that not even continuous(or bounded) potentials suffice, and we imposed some extra decay condition near $0$, mentioned in Theorem \ref{mainthm1}. 
\end{remark}
\begin{remark} \label{rem:rem3}
It would of course be very convenient, if Theorem \ref{mainthm1} continues to hold for any $q\in L^1_{\loc} ([0,\infty))$. However, it seems that to treat the logarithmic singularities e.g. in $I_3$ and $I_6$, one has to work with H\"older's inequality, which of course isn't available for locally integrable potentials.
\end{remark}
If we are now able to prove that, in addition to Theorem \ref{thm:estu}, $u$ is also a $C^2$ function, then we can indeed conclude that it satisfies \eqref{eq:u}--\eqref{eq:ubc}. This will be discussed next:
\begin{lemma} \label{lem:Bder}
Let $q\in C^{1}([0, L])$. Then $B(x,\cdot) \in C^2([0,x])$.
\end{lemma}
\begin{proof}
This proof closely follows the arguments from \cite{soh}(cf. end of page 6). Let the corresponding kernel $B(x,y)$ be given by \eqref{eq:Btr}. We start by establishing an integral equation for $B$. We thus introduce the new coordinates $\ti{z}:=\sqrt{z}=\frac{x+y}{2}, \quad \ti{s}:=\sqrt{s}=\frac{x-y}{2},$ and the function $\ti{u}(\ti{z},\ti{s}):=B(x,y)=B(\ti{z}+\ti{s},\ti{z}-\ti{s})$, so that \eqref{eq:B2} transforms to 
\[
\frac{\partial^2 \ti{u}}{\partial \ti{z} \partial \ti{s}}+\frac{ 4l(l+1)\ti{z}\ti{s}}{(\ti{z}^2-\ti{s}^2)^2}\ti{u}=-q(\ti{z}+\ti{s})\ti{u}.
\] 
Now we integrate with respect to $\ti{z}$ and $\ti{s}$ and transform back to $x$ and $y$ coordinates($\ti{x}=\ti{z}+\ti{s}, \quad \ti{y}=\ti{z}-\ti{s}$) and obtain the following equation for $B$:
\begin{align*}
B(x,y)= &\int_{0}^{\frac{x+y}{2}} q(\ti{x})d \ti{x} +\frac12 \left( \int_{0}^{\frac{x+y}{2}} d \ti{x} \int_0^{\ti{x}} +\int_{\frac{x+y}{2}}^{x} d \ti{x} \int_{0}^{x+y-\ti{x}} \right) \\
& \times \left[ q(\ti{x})+l(l+1)\left(\frac{1}{\ti{x}^2} - \frac{1}{\ti{y}^2} \right) \right] B(\ti{x}, \ti{y}) d\ti{y}, \quad 0<y \le x.
\end{align*}
This immediately shows that $B$ obtains second derivatives, if $q$ is differentiable. 
\end{proof}
So far we have shown that for a smooth potential $q$ the transformation operators exist. Now we suppose the assumptions on $q$ from Theorem \ref{mainthm1} and proceed as follows: Approximate the function $q$ by a sequence of smooth functions $q_n$, such that for any $x\in(0,L]$, $q_n$ converges to $q$ in the $L^p((0,x])$-norm(or $L^p((0,x], z^{-\frac{1}{p'}})$-norm in the $-\frac12$-case). Let $B_n(x,y)$, $B(x,y)$ be the kernels obtained from the potentials $q_n$, $q$ respectively via Theorem \ref{mainthm1}. Then from \eqref{eq:GLest} we can conclude that $B_n$ converges to $B$ uniformly on $[0,L]^2$. This proves Theorem \ref{mainthm1}. 
\section{Transformation Operators near $\infty$}
Completely analogous computations as in the beginning of the previous section lead to the following set of equations for the transformation operator $K$:
\begin{align} \label{eq:K}
&\left( \frac{\partial^2 }{\partial x^2} - \frac{\partial^2 }{\partial y^2}  +  \frac{l(l+1)}{y^2} - \frac{l(l+1)}{x^2}-q(x) \right)K(x,y)=0, \quad 0<x<y \\
&\quad \frac{\partial K(x,x)}{\partial x}=-\frac{q(x)}{2}, \quad \lim_{y \to \infty} K(x,y) =0=\lim_{y \to \infty} \frac{\partial K(x,y)}{\partial y}. \label{eq:Kbc}
\end{align}
The next step is to put the problem into integral form using the same change of variables as in Lemma \ref{lem:Bder}, i.e. $\xi:=\frac{x+y}{2}, \eta:=\frac{y-x}{2}, \quad w(\xi,\eta):=K(x,y)=K(\xi-\eta,\xi+\eta)$, so that \eqref{eq:K} transforms to 

\begin{align}
\frac{\partial^2 w}{\partial \xi \partial \eta}&+\frac{ 4l(l+1)\xi \eta}{(\xi^2-\eta^2)^2}w=-q(\xi-\eta)w \label{eq:w} \\
w(\xi,0)=\frac12 \int_{x}^{\infty} &q(z) dz, \quad \lim_{\xi \to \infty} w(\xi,\eta)=0, \quad \eta>0. \label{eq:Kbc}
\end{align}
Again, as in the previous section, for the time being we assume $q$ to be differentiable. We furthermore introduce the Riemann function $v_3$ as a solution to the problem

\begin{align} \label{eq:v3}
\frac{\partial^2 v_3}{\partial z \partial s}+\frac{ 4l(l+1)zs}{(z^2-s^2)^2}v_3&=0, \quad 0<s<\eta<\xi<z<\infty \\
v_3(z,s;\xi,\eta) \Big|_{z=\xi}&=1 \quad s\in[0,\eta], \nn \\
v_3(z,s;\xi,\eta) \Big|_{s=\eta}&=1 \quad z\in[\xi,\infty). \nn
\end{align}
Using the transformation $\ti{z}=z^2$, $\ti{s}=s^2$ and defining $\ti{v_3}:=(\ti{z}-\ti{s})^l v_3$, we see that $\ti{v_3}$ satisfies the equation $L\ti{v_3}=0$( $L$ is again defined by \eqref{def:L}). Similar considerations as for $v_1$ lead to the following explicit formula for $v_3$:
\be \label{eq:formulav_3}
v_3(z,s; \eta, \xi)= \left( \frac{z^2-\eta^2}{z^2-s^2} \cdot \frac{\xi^2-s^2}{\xi^2-\eta^2} \right)^l \hyp21{-l,-l}{1}{\frac{z^2-\xi^2}{z^2-\eta^2}\cdot\frac{\eta^2-s^2}{\xi^2-s^2}}.
\ee
Let us proceed with a more detailed analysis of $v_3$:
\begin{lemma} \label{lem:v3est}
The Riemann function $v_3$ satisfies the following estimate:
\begin{align} \label{eq:estv_3}
\abs{v_3(z,s;\xi,\eta)} &\le C_l \left( \frac{\xi^2}{\xi^2-\eta^2} \right)^l , \quad l>-\frac12 \\
\abs{v_3(z,s;\xi,\eta)} &\le \frac{C_{-\frac12}}{\beta}  \left( \frac{\xi^2}{\xi^2-\eta^2} \right)^{-\frac12+\beta} ,  \quad l=-\frac12
\end{align}
where $0<s<\eta<\xi<z<\infty$ and $0<\beta\le \frac12$.
\end{lemma}
\begin{proof}
Set $t:=\frac{z^2-\xi^2}{z^2-\eta^2}\cdot\frac{\eta^2-s^2}{\xi^2-s^2}$. Clearly $0<t<1$. A short calculation also shows $1-t=\frac{z^2-s^2}{z^2-\eta^2} \cdot \frac{\xi^2-\eta^2}{\xi^2-s^2}$, thus we also have $0<1-t<1$. It remains to look at asymptotics for $(1-t)^{-l} \hyp21{-l,-l}{1}{t}$, where $0 \le t \le 1$. In the case $l \ne -\frac12$, by employing \eqref{eq:Fnear11} we obtain
\[
\abs{(1-t)^{-l} \hyp21{-l,-l}{1}{t}} \le C_l (1-t)^{-l}
\]
which immediately leads to \eqref{eq:estv_3}. Moreover, in the case $l=-\frac12$, using \eqref{eq:Fnear13}, we get
\[
(1-t)^{\frac12} \hyp21{\frac12,\frac12}{1}{t}=\OO\left((1-t)^{\frac12} \log\left(\frac{1}{1-t}\right) \right), \quad t \to 0,
\]
and thus again \eqref{eq:estv_3} follows, since on the intervall $[0,1]$ the expression $(1-t)^{\beta} \log\left(\frac{1}{1-t}\right)$ is a constant multiple of $\frac{1}{\beta}$. 
\end{proof}
Next, if we apply Riemann's method to \eqref{eq:w}, we end up with the subsequent integral equation for $w$:
\begin{align} \label{eq:intequinfty}
w(\xi,\eta)=\frac12 \int_{\xi}^{\infty} v_3(z,0;\xi,\eta)q(z)dz+\int_{\xi}^{\infty} \int_{0}^{\eta} q(z-s)v_3(z,s;\xi,\eta) w(z,s) ds dz. 
\end{align}
For simplicity let us start with the case $l>-\frac12$.
Instead of $w$, we will consider an integral equation for the function $\ti{w}:= \left( \frac{\xi^2}{\xi^2-\eta^2} \right)^{-l} w(\xi,\eta)$. Thus in the sequel we are concerned with the following expression:
\begin{align*}
\ti{w}(\xi,\eta)&=\frac12 \left( \frac{\xi^2}{\xi^2-\eta^2} \right)^{-l} \int_{\xi}^{\infty} v_3(z,0;\xi,\eta)q(z)dz \\
&+\int_{\xi}^{\infty} \int_{0}^{\eta} q(z-s)v_3(z,s;\xi,\eta) \ti{w}(z,s) ds dz. 
\end{align*}
Again we intend to apply successive approximation and set $\ti{w}=\sum_{n=0}^{\infty} \ti{w}_n$, where the $\ti{w}_n$'s are defined recursively as follows:
\begin{align} 
\ti{w}_0(\xi,\eta)&:= \left( \frac{\xi^2}{\xi^2-\eta^2} \right)^{-l} \frac12 \int_{\xi}^{\infty} v_3(z,0;\xi,\eta)q(z)dz \nn \\
\ti{w}_{n+1}(\xi,\eta)&:= \left( \frac{\xi^2}{\xi^2-\eta^2} \right)^{-l}\int_{\xi}^{\infty} \int_{0}^{\eta} q(z-s)v_3(z,s;\xi,\eta) \ti{w}_n(z,s) ds dz. \label{eq:wn}
\end{align}
In the case $l=-\frac12$, we will consider $\ti{w}:= \left( \frac{\xi^2}{\xi^2-\eta^2} \right)^{\frac12-\beta} w(\xi,\eta)$ instead, the definitions in \eqref{eq:wn} will change in an obvious way. 
We will end up with the following theorem:
\begin{theorem} \label{thm:estw}
Under the conditions on $q$ stated in Theorem \ref{mainthm2}, there is a unique continuous function  $w(\xi,\eta)$ that solves \eqref{eq:intequinfty} and satisfies  
\begin{align}
	\abs{w(\xi,\eta)} &\le\frac{C_l}{2} \left( \frac{\xi^2}{\xi^2-\eta^2} \right)^l \ti{\sigma}_0(\xi)\exp \left(C_l[\ti{\sigma}_1(\xi-\eta)-\ti{\sigma}_1(\xi)]\right), \quad l>-\frac12 \nn \\
	\abs{w(\xi,\eta)} &\le \frac{C_{-\frac12}}{2\beta}  \left( \frac{\xi^2}{\xi^2-\eta^2} \right)^{-\frac12+\beta}
 \ti{\sigma}_0(\xi)\exp \left(\frac{C_{-\frac12}}{\beta}[\ti{\sigma}_1(\xi-\eta)-\ti{\sigma}_1(\xi)]\right), \quad l=-\frac12, \label{eq:west}
\end{align}
where $0<\beta\le \frac12$.
\end{theorem}
To this end we need to find suitable estimates for the iterates $w_n$, which is done in the subsequent lemma:
\begin{lemma} \label{lem:estwn}
In the case $l>-\frac12$, we have the following estimates for our iterates $\ti{w_n}$ defined in \eqref{eq:wn}:
\begin{align*}
\abs{\ti{w_0}(\xi,\eta)} &\le \frac{C_l}{2}  \ti{\sigma}_0(\xi) 
\end{align*}

\begin{align*}
\abs{\ti{w_1}(\xi,\eta)} &\le \frac{C_l}{2} \ti{\sigma}_0(\xi)C_l[\ti{\sigma}_1(\xi-\eta)-\ti{\sigma}_1(\xi)]
\end{align*}
and finally
\begin{align*}
\abs{\ti{w_n}(\xi,\eta)} &\le 
 \frac{C_l}{2}  \ti{\sigma}_0(\xi)\frac{\left( C_l[\ti{\sigma}_1(\xi-\eta)-\ti{\sigma}_1(\xi)]\right)^n}{n!}
\end{align*}
\end{lemma}
\begin{proof}
The estimate for $\ti{w_0}$(or $w_0$ resp.) follows immediately from \eqref{eq:wn} and \eqref{eq:estv_3}. Let's proceed with $\ti{w_1}$: 
\begin{align*}
\abs{\ti{w_1}(\xi,\eta)} &\le \left( \frac{\xi^2}{\xi^2-\eta^2} \right)^{-l} \int_{\xi}^{\infty} \int_{0}^{\eta} \abs{q(z-s)v_3(z,s;\xi,\eta) w_0(z,s)} ds dz \\
&\le  \left( \frac{\xi^2}{\xi^2-\eta^2} \right)^{-l} \int_{\xi}^{\infty} \int_{0}^{\eta} \frac{C_l^2}{2} \left( \frac{\xi^2}{\xi^2-\eta^2} \right)^{l} \ti{\sigma}_0(z) \abs{q(z-s)} ds dz \\
&\le \frac{C_l^2}{2}  \ti{\sigma}_0(\xi) \int_{\xi}^{\eta} (\ti{\sigma_0}(z-\eta)-\ti{\sigma_0}(z))dz \\
&=\frac{C_l^2}{2} \ti{\sigma}_0(\xi)(\ti{\sigma_1}(\xi-\eta)-\ti{\sigma_1}(\xi)).
\end{align*}
Now we come to the estimate for $\ti{w_n}$, which is done inductively:
\begin{align*}
\abs{\ti{w_n}(\xi,\eta)} &\le \left( \frac{\xi^2}{\xi^2-\eta^2} \right)^{-l} \int_{\xi}^{\infty} \int_{0}^{\eta} \abs{q(z-s)v_3(z,s;\xi,\eta) \ti{w_{n-1}}(z,s)} ds dz \\
&\le \left( \frac{\xi^2}{\xi^2-\eta^2} \right)^{-l} \int_{\xi}^{\infty} \int_{0}^{\eta} \frac{C_l^2}{2} \left( \frac{\xi^2}{\xi^2-\eta^2} \right)^{l} \ti{\sigma}_0(z) \\
&\times \abs{q(z-s)} \frac{(C_l(\ti{\sigma_1}(z-s)-\ti{\sigma_1}(z)))^{n-1}}{(n-1)!} ds dz \\
&\le \frac{C_l^2}{2} \ti{\sigma}_0(\xi) \int_{\xi}^{\infty} \frac{(C_l(\ti{\sigma_1}(z-\eta)-\ti{\sigma_1}(z)))^{n-1}}{(n-1)!} (\ti{\sigma_1}(z-\eta)-\ti{\sigma_1}(z)) dz \\
&=\frac{C_l^2}{2} \ti{\sigma}_0(\xi) \frac{(C_l(\ti{\sigma_1}(\xi-\eta)-\ti{\sigma_1}(\xi)))^{n}}{n!}.
\end{align*}
\end{proof}

Minor modifications have to be made in the case $l=-\frac12$, namely, the expression $\left( \frac{\xi^2}{\xi^2-\eta^2} \right)^{-l}$ has to be replaced by $ \left( \frac{\xi^2}{\xi^2-\eta^2} \right)^{\frac12-\beta}$ and $C_l$ by $\frac{C_{-\frac12}}{\beta}$.
\begin{proof}[Proof of Theorem~\ref{thm:estw}]
Everything now follows from the Lemma \ref{lem:estwn} again by uniform convergence of the corresponding series $\ti{w}=\sum_{n=0}^{\infty} \ti{w}_n$(or $w=\sum_{n=0}^{\infty} w_n$ respectively).
\end{proof}
The next result is the analogous version of Lemma \ref{lem:Bder}
\begin{lemma} \label{lem:Kder}
Let $q$ satisfy the assumptions of Theorem \ref{mainthm2} and additionally $q\in C^{1}([x, \infty))$. Then $K(x,.) \in C^2([x,\infty))$.
\end{lemma}
\begin{proof}
Similar ideas as in the previous section lead to the following integral equation for $K(x,y)$:
\begin{align*}
K(x,y)= &\int_{\frac{x+y}{2}}^{\infty} q(\ti{x})d \ti{x} +\frac12 \left( \int_{x}^{\frac{x+y}{2}} d \ti{x} \int_{-\ti{x}+x+y}^{\ti{x}+y-x} +\int_{\frac{x+y}{2}}^{\infty} d \ti{x} \int_{\ti{x}}^{\ti{x}+y-x} \right) \\
& \times \left[ q(\ti{x})+l(l+1)\left(\frac{1}{\ti{x}^2} - \frac{1}{\ti{y}^2} \right) \right] K(\ti{x}, \ti{y}) d\ti{y}, \quad 0<x \le y,
\end{align*}
which gives the desired claim.
\end{proof}
The arguments for the approximation procedure, that conclude the proof of Theorem \ref{mainthm2}, are now exactly the same as in the previous section.
\label{ll-sec}

\appendix

\section{The Gauss Hypergeometric function}\label{sec:hypgeom}
Here we collect basic formulas and information on the Gauss hypergeometric function (see, e.g., \cite{aar},\cite{dlmf}). First of all by $\Gamma$ is denoted the classical gamma function \dlmf{5.2.1}. For $x\in\C$ and $n\in\N_0$
\[
(x)_n:=x(x+1)\cdots(x+n-1)\quad(n>0),\quad (x)_0:=1;\quad
\binom{n+x}{n} := \frac{(x+1)_n}{n!}
\]
denote the Pochhammer symbol \dlmf{5.2.4}
and the binomial coefficient, respectively. Notice that for $-x\notin \N_0$
\[
(x)_n = \frac{\Gamma(x+n)}{\Gamma(x)},\qquad 
\binom{n+x}{n} = \frac{\Gamma(x+n+1)}{\Gamma(x+1)\Gamma(n+1)}\,.
\]
Moreover, the above formulas allow to define the Pochhammer symbol and the binomial coefficient for noninteger $x$, $n>0$ as well. 
For $-c \notin \N_0$
the Gauss hypergeometric function \dlmf{15.2.1} is defined by 
\be\label{Gausshyp}
\hyp21{a,b}{c}{z} := \sum_{k=0}^\infty \frac{(a)_k(b)_k}{(c)_k k!}\,z^k
\quad\mbox{($|z|<1$ or else $-a$ or $-b\in \N_0$).}
\ee 
The branch cut is chosen along the positive real axis.
By analytic continuation this definition may also be extended to other values of $z$. Thus the derivative is also easy to compute and given by the following formula(see \dlmf{15.5.1}):

\be \label{eq:Gausshypder}
\frac{\partial}{\partial z} \hyp21{a,b}{c}{z}=\frac{ab}{c} \hyp21{a+1,b+1}{c+1}{z}.
\ee
Functions of the form \eqref{Gausshyp} are closely related to the hypergeometric equation
\be \label{eq:hypeq}
x(1-x) \frac{d^2 f}{dx^2} +(c-(a+b+1)x) \frac{df}{dx}-abx=0.
\ee
Indeed, the hypergeometric functions appear in explicit formulas for solutions of \eqref{eq:hypeq}, one has to be careful with certain values of the parameters $a$, $b$ and $c$ though. For a summary of the types of solutions that may occur, see \dlmf{15.10}. Next, we also need the asymptotic behavior near the possible singular points $1$ and $\infty$ for $\hyp21{a,b}{c}{z}$ for specific values of $a$, $b$ and $c$(see \dlmf{15.4.20,15.4.21,15.8.8}):

\be \label{eq:Fnear11}
\hyp21{a,b}{c}{1}=\frac{\Gamma(c) \Gamma(c-a-b)}{\Gamma(c-a)\Gamma(c-b)}, \quad \re(c-a-b)>0,
\ee
\be \label{eq:Fnear13}
\lim_{z \to 1-} \frac{\hyp21{a,b}{a+b}{z}}{-\log(1-z)}=\frac{ \Gamma(a+b)}{\Gamma(a)\Gamma(b)},
\ee
and
\begin{align} \label{eq:Fnearinfty}
\hyp21{a,a}{c}{z}=\frac{\Gamma(c) (-z)^{-a}}{\Gamma(a)} \sum_{k=0}^{\infty}&\frac{(a)_n}{(k!)^2 \Gamma(c-a-k)}(-1)^k z^{-k} \\
&\times (\log(-z)+2\psi(k+1)-\psi(a+k)-\psi(c-a-k)), \nonumber
\end{align}
if $|z|>1$. Here $\psi$ denotes the digamma function \dlmf{5.2.2}. It satisfies the reflection formula (c.f. \dlmf{5.5.2} )
\be \label{eq:psiref}
\psi(z+1)=\psi(z)+\frac1z
\ee
 and we also briefly mention an estimate near $\infty$ (c.f. \dlmf{5.11.2} ):
\be \label{eq:psinearinfty}
\psi(z)=\log{z}-\frac{1}{2z}+\OO(z^{-2}), \quad z \to \infty,
\ee
which will be useful in order to show absolute convergence of the series in \eqref{eq:Fnearinfty}. 
To conclude, we also need to mention that for integer values of $a$, the hypergeometric function reduces to a polynomial.

\bigskip
\noindent
{\bf Acknowledgments.} The author is very grateful to Aleksey Kostenko and Gerald Teschl for introducing the topic to him and for helpful discussions. He also thanks Vladislav Kravchenko and Sergii Torba for providing him with the paper \cite{soh}, and Iryna Egorova for the copy of A.\ S.\ Sohin's PhD thesis. He also wants to thank the anonymous referees for the careful reading of the manuscript and for pointing out issues related to a new version of \cite{kty}.


\end{document}